\documentclass[reqno, 12pt]{amsart}

\pdfoutput=1
\makeatletter
\let\origsection=\section \def\section{\@ifstar{\origsection*}{\mysection}} 
\def\mysection{\@startsection{section}{1}\z@{.7\linespacing\@plus\linespacing}{.5\linespacing}{\normalfont\scshape\centering\S}}
\makeatother    

\usepackage{amsmath,amssymb,amsthm}
\usepackage[foot]{amsaddr}
\usepackage{mathrsfs}
\usepackage{mathtools}
\usepackage{mathabx}\changenotsign
\usepackage{dsfont}
\usepackage{nicefrac}
\usepackage[normalem]{ulem}
\usepackage{graphicx}
\usepackage{xcolor}
\usepackage{tikz}
\usepackage[foot]{amsaddr}
\usepackage{verbatim}
\usepackage{thm-restate}

\usepackage{enumitem}

\usepackage{hyperref}
\usepackage{cleveref}
\hypersetup{
    colorlinks,
    linkcolor={red!60!black},
    citecolor={green!60!black},
    urlcolor={blue!60!black},
}

\usepackage[open,openlevel=2,atend]{bookmark}

\input{_preamble/myBiblatex}
\usepackage[T1]{fontenc}
\usepackage{lmodern}
\usepackage[babel]{microtype}
\usepackage[english]{babel}

\usepackage{geometry}
\numberwithin{equation}{section}
\numberwithin{figure}{section}

\let\polishlcross=\l
\def\l{\ifmmode\ell\else\polishlcross\fi}

\def\paragraph#1{%
  \smallskip
  \noindent\textbf{#1.}\enspace}

\makeatletter
\def\moverlay{\mathpalette\mov@rlay}
\def\mov@rlay#1#2{\leavevmode\vtop{   \baselineskip\z@skip \lineskiplimit-\maxdimen
   \ialign{\hfil$\m@th#1##$\hfil\cr#2\crcr}}}
\newcommand{\charfusion}[3][\mathord]{
    #1{\ifx#1\mathop\vphantom{#2}\fi
        \mathpalette\mov@rlay{#2\cr#3}
      }
    \ifx#1\mathop\expandafter\displaylimits\fi}
\makeatother

\usepackage{thmtools, thm-restate}

\theoremstyle{plain}
\newtheorem{thm}{Theorem}[section]
    \crefname{thm}{Theorem}{Theorems}
\newtheorem{theorem}[thm]{Theorem}
    \crefname{theorem}{Theorem}{Theorems}
\newtheorem{lemma}[thm]{Lemma}
    \crefname{lemma}{Lemma}{Lemmas}

    \crefname{lem}{Lemma}{Lemmas}
\newtheorem{corollary}[thm]{Corollary}
    \crefname{corollary}{Corollary}{Corollaries}

    \crefname{cor}{Corollary}{Corollaries}
\newtheorem{proposition}[thm]{Proposition}
    \crefname{proposition}{Proposition}{Propositions}

    \crefname{prop}{Proposition}{Propositions}

    \crefname{problem}{Problem}{Problems}
\newtheorem{conjecture}[thm]{Conjecture}
    \crefname{conjecture}{Conjecture}{Conjectures}
\newtheorem{observation}[thm]{Observation}
    \crefname{observation}{Observation}{Observations}
\newtheorem{question}[thm]{Question}
    \crefname{question}{Question}{Questions}
\newtheorem*{claim*}{Claim}
\newtheorem{claim}{Claim}[]
    \crefname{claim}{Claim}{Claims}

    \crefname{clm}{Claim}{Claims}
\newtheorem*{case*}{Case}

    \crefname{case}{Case}{Case}
\newtheorem{thm-intro}{Theorem}[]
    \crefname{thm-intro}{Theorem}{Theorems}
\newtheorem{conj-intro}[thm-intro]{Conjecture}
    \crefname{conj-intro}{Conjecture}{Conjectures}
\newtheorem{question-intro}[thm-intro]{Question}
    \crefname{question-intro}{Question}{Questions}

\theoremstyle{definition}

    \crefname{definition}{Definition}{Definitions}

    \crefname{remark}{Remark}{Remarks}

    \crefname{remarks}{Remarks}{Remarks}

    \crefname{situation}{Situation}{Situations}

    \crefname{construction}{Construction}{Constructions}

    \crefname{construction}{Example}{Examples}
\newtheorem*{example*}{Example}

\newenvironment{subproof}[1][Proof.]{%
    \begin{proof}[{#1}]%
        }{%
    \end{proof}}

\DeclareFontFamily{U}  {MnSymbolC}{}
\DeclareSymbolFont{MnSyC}         {U}  {MnSymbolC}{m}{n}
\DeclareFontShape{U}{MnSymbolC}{m}{n}{
    <-6>  MnSymbolC5
   <6-7>  MnSymbolC6
   <7-8>  MnSymbolC7
   <8-9>  MnSymbolC8
   <9-10> MnSymbolC9
  <10-12> MnSymbolC10
  <12->   MnSymbolC12}{}
\DeclareMathSymbol{\powerset}{\mathord}{MnSyC}{180}

\let\emptyset=\varnothing
\let\setminus=\smallsetminus

\newcommand*{\abs}[1]{\ensuremath{{\left\lvert {#1} \right\rvert}}}

\DeclareMathOperator{\dist}{dist}
\DeclareMathOperator{\rt}{root}
\DeclareMathOperator{\lv}{lv}
\DeclareMathOperator{\supp}{sp}
\newcommand{\tin}{\mathsf{tree} \textnormal{-} \alpha}

\hypersetup{
    pdftitle={A coarse Erd\H{o}s-P\'{o}sa theorem},
    pdfauthor={Jungho Ahn, J.~Pascal Gollin, Tony Huynh, O-joung Kwon}
}

\begin{document}

\author[J.~Ahn]{Jungho Ahn$^1$}
\address{$^1$Korea Institute for Advanced Study, Seoul, South~Korea}
\email{junghoahn@kias.re.kr}

\author[J.~P.~Gollin]{J.~Pascal Gollin$^2$}
\address{$^2$FAMNIT, University of Primorska, Koper, Slovenia}
\email{\tt pascal.gollin@famnit.upr.si}

\author[T.~Huynh]{Tony Huynh$^3$}
\address{$^3$Department of Computer Science, Sapienza University of Rome, Rome, Italy}
\email{\tt tony.bourbaki@gmail.com}

\author[O.~Kwon]{O-joung Kwon$^{4,5}$}
\address{$^4$Department of Mathematics, Hanyang University, Seoul, South Korea}
\address{$^5$Discrete Mathematics Group, Institute for Basic Science (IBS), Daejeon, South Korea}
\email{ojoungkwon@hanyang.ac.kr}

\title{A coarse Erd\H{o}s-P\'{o}sa theorem}
\thanks{%
All authors are corresponding authors.
The first author was supported by the KIAS Individual Grant (CG095301) at Korea Institute for Advanced Study. %
The second author was supported by the Institute for Basic Science (IBS-R029-Y3) and in part by the Slovenian Research and Innovation Agency (research project N1-0370). %
The third author was partially supported by Progetti Grandi di Ateneo Programme, Grant RG123188B3F7414A (ASGARD). %
The fourth author is supported by the National Research Foundation of Korea (NRF) grant funded by the Ministry of Science and ICT (No. NRF-2021K2A9A2A11101617 and No. RS-2023-00211670) and by the Institute for Basic Science (IBS-R029-C1). %
}

\date{January 10, 2025}

\begin{abstract}
    An \emph{induced packing} of cycles in a graph is a set of vertex-disjoint cycles with no edges between them.
    We generalise the classic Erd\H{o}s-P\'osa theorem to induced packings of cycles.
    More specifically, we show that there exist functions~${f(k,\ell) = \mathcal{O}(\ell k\log k)}$ and~${g(k)=\mathcal{O}(k\log k)}$ such that for all integers~${k\geq1}$ and ${\ell\geq3}$, every graph~$G$ contains either an induced packing of~$k$ cycles of length at least~$\ell$, not necessarily induced cycles, or sets~$X_1$ and~$X_2$ of vertices with $\abs{X_1}\leq f(k,\ell)$ and $\abs{X_2}\leq g(k)$ such that, after removing the closed neighbourhood of~$X_1$ or the ball of radius~$\ell$ around~$X_2$, the resulting graph has no cycle of length at least~$\ell$ in~$G$.
    Our proof is constructive and yields a polynomial-time algorithm finding either the induced packing or the sets~$X_1$ and~$X_2$ when~$\ell$ is a constant.
    Furthermore, we show that for every positive integer~$d$, if a graph~$G$ does not contain two cycles at distance more than~$d$, then~$G$ contains sets~$X_1$ and~$X_2$ of vertices with $\abs{X_1}\leq12(d+1)$ and $\abs{X_2}\leq12$ such that, after removing the ball of radius~$2d$ around~$X_1$ or the ball of radius~$3d$ around~$X_2$, the resulting graphs are forests.

    As a corollary, we prove that every graph with no $K_{1,t}$ induced subgraph and no induced packing of~$k$ cycles of length at least~$\ell$ has tree-independence number at most $\mathcal{O}(t\ell k\log k)$, and one can construct a corresponding tree-decomposition in polynomial time when~$\ell$ is a constant.
    This resolves a special case of a conjecture of Dallard et al. (arXiv:2402.11222), and implies that on such graphs, many NP-hard problems, such as \textsc{Maximum Weight Independent Set}, \textsc{Maximum Weight Induced Matching}, \textsc{Graph Homomorphism}, and \textsc{Minimum Weight Feedback Vertex Set}, are solvable in polynomial time. On the other hand, we show that the class of all graphs with no $K_{1,3}$ induced subgraph and no two cycles at distance more than~$2$ has unbounded tree-independence number. 
\end{abstract}


\maketitle

\section{Introduction}

The classic theorem of Erd\H{o}s and P\'{o}sa~\cite{ErdosP1965} asserts that there exists a function \linebreak ${f(k) = \mathcal{O}(k \log k)}$ such that for every positive integer~$k$, every graph~$G$ contains either~$k$ vertex-disjoint cycles or a set~$X$ of at most~$f(k)$ vertices such that~$X$ intersects all cycles in~$G$. 
This result has motivated a long line of research for finding similar dualities for other combinatorial objects.
These are too numerous to list here, but we refer the interested reader to the survey of Raymond and Thilikos~\cite{RaymondT2017}. 

Recently, induced variants of Menger's theorem have received considerable interest~\cite{AlbrechtsenHTJKW2024,GeorgakopoulosP2023,GartlandKL2023,HendreyNST2023,NguyenSS2024-a}. 
For a set~$X$ of vertices in a graph~$G$, let $B_G(X,d)$ be the set of vertices in~$G$ whose distance from~$X$ is at most~$d$. 
Albrechtsen, Huynh, Jacobs, Knappe, and Wollan~\cite{AlbrechtsenHTJKW2024} and independently Georgakopoulos and Papasoglu~\cite{GeorgakopoulosP2023} proved that there exists a function~${f(d) = \mathcal{O}(d)}$ such that for every positive integer~${d}$, every graph~$G$, and every pair $(X,Y)$ of subsets of~$V(G)$, $G$ contains either two ${(X,Y)}$-paths which are at distance at least~$d$ or a vertex~$v$ such that~${B_G(\{v\},f(d))}$ hits all ${(X,Y)}$-paths. 
Both~\cite{AlbrechtsenHTJKW2024} and~\cite{GeorgakopoulosP2023} conjectured that this can be generalised so that one can find either~$k$ ${(X,Y)}$-paths which are pairwise at distance at least~$d$ or a set~$S$ of at most~${k-1}$ vertices such that~${B_G(S,f(d))}$ hits all ${(X,Y)}$-paths. 
Nguyen, Scott, and Seymour~\cite{NguyenSS2024-a} disproved this conjecture for~${(d,k) = (3,3)}$, but it is still open for ${d = 2}$ and general~$k$. 
They also proposed a new conjecture by relaxing the size of~$S$: 
there are functions~$f$ and~$g$ such that for every pair~$(k,d)$ of positive integers and every graph~$G$ with vertex sets~$X$ and~$Y$, $G$ contains either~$k$ $(X,Y)$-paths which are pairwise at distance at least~$d$ or a set~$S$ of at most~${g(k,d)}$ vertices such that~${B_G(S,f(k,d))}$ hits all ${(X,Y)}$-paths. 
This conjecture still remains open. 

Motivated by these results, we ask whether there is an induced packing version of the classic theorem of Erd\H{o}s and P\'{o}sa for cycles.  
A set~$\mathcal{F}$ of vertex-disjoint cycles in a graph~$G$ is an \emph{induced packing} of~$G$ if there is no edge of~$G$ between any two distinct cycles in~$\mathcal{F}$.
Induced packings of cycles~\cite{AtminasKR2016,BonnetDGTW2023,Bonamy2024,NguyenSS2024-b,ChudnovskySSS2023} or odd cycles~\cite{GolovachKPT2012,Bonamy2021,DvorakP2023} have already been considered in the literature.
The natural analogue of the Erd\H{o}s-P\'{o}sa theorem for induced packings is whether there are a function~$f$ and a constant~$c$ such that whenever~$G$ has no induced packing of~$k$ cycles, there is a set~$X$ of at most~${f(k)}$ vertices such that~${B_G(X,c)}$ hits all cycles of~$G$.
Observe that~$c$ should be at least~$1$, because complete graphs have no induced packing of two cycles, but the class of complete graphs has arbitrary large feedback vertex set number. 

Our first main result answers the above question affirmatively with a best possible~$f$ and~${c=1}$.
We actually prove a stronger statement considering long cycles.

\begin{restatable}{theorem}{mainone}
    \label{thm:main1}
    There exist functions~${f(k,\ell) = \mathcal{O}(\ell k\log k)}$ and~${g(k)=\mathcal{O}(k\log k)}$ such that for all integers ${k\geq1}$ and ${\ell\geq3}$ and every graph~$G$, one can find in $\abs{V(G)}^{\mathcal{O}(\ell)}$ time either an induced packing of~$k$ cycles of length at least~$\ell$ or sets~$X_1$ and~$X_2$ of vertices with $\abs{X_1}\leq f(k,\ell)$ and $\abs{X_2}\leq g(k)$ such that neither~${G-B_G(X_1,1)}$ nor ${G-B_G(X_2,\ell)}$ has a cycle of length at least~$\ell$.
\end{restatable}

We remark that the long cycles in the theorem might not be induced cycles.
Here is a proof that the functions~$f$ and~$g$ in Theorem~\ref{thm:main1} are best possible up to a multiplicative constant for any fixed integer~$\ell\geq3$.
Erd\H{o}s and P\'osa~\cite{ErdosP1965} showed that there is no function~${f(k) = o(k \log k)}$ that works for the original Erd\H{o}s-P\'osa theorem. 
For an arbitrary graph~$G$, let~$H$ be the graph obtained from~$G$ by subdividing each edge~${2\ell}$ times.
Then every cycle in~$H$ is of length at least $3(2\ell+1)>\ell$, and every set of vertex-disjoint cycles in~$H$ is an induced packing of cycles in~$H$.
Note that~$H$ has an induced packing of~$k$ cycles if and only if~$G$ has~$k$ vertex-disjoint cycles.
In addition, all of the following are equal to each other:
\begin{itemize}
    \item the feedback vertex set number of~$G$,
    \item the minimum size of a set~${X_1 \subseteq V(H)}$ such that $B_H(X_1,1)$ hits all cycles in~$H$, and
    \item the minimum size of a set~${X_2 \subseteq V(H)}$ such that $B_H(X_2,\ell)$ hits all cycles in~$H$.
\end{itemize}
Thus, all functions~$f$ and~$g$ that work for Theorem~\ref{thm:main1} also work for the original Erd\H{o}s-P\'osa theorem.
Therefore, for any fixed integer~$\ell\geq3$, there are no functions $f(k,\ell)=o(k\log k)$ and $g(k)=o(k\log k)$ that work for Theorem~\ref{thm:main1}.
The above reduction also shows that Theorem~\ref{thm:main1} implies the original Erd\H{o}s-P\'osa theorem. 
It is open whether there is a function $f(k,\ell)=o(\ell k\log k)$ that works for Theorem~\ref{thm:main1} when we regard~$\ell$ as a variable.

We say that a graph is \emph{$H$-free} if it does not contain an induced subgraph isomorphic to~$H$. 
Theorem~\ref{thm:main1} implies that every $K_{1,t}$-free graph having no large induced packing of cycles has bounded tree-independence number.
Recently, Yolov~\cite{DBLP:conf/soda/Yolov18} and Dallard et al.~\cite{DallardMS2024} independently introduced this parameter, and it has already garnered significant interest~\cite{DallardMS2024,DallardKKMMW2024,AbrishamiACHSV2024,LimaMMIRS2024,ChudnovskyHLS2024,ChudnovskyHT2024,DBLP:conf/soda/Yolov18,LimaMMIRS2024}. 
A graph~$H$ is an \emph{induced minor} of a graph~$G$ if~$H$ is isomorphic to a graph~$H'$ obtained from~$G$ by contracting edges and deleting vertices.
We say that~$G$ is \emph{$H$-induced-minor-free} if~$G$ does not contain~$H$ as an induced minor. 
Dallard et al.~\cite[Conjecture 7.2]{DallardKKMMW2024} conjectured that for every planar graph~$H$, the class of $K_{1,t}$-free and $H$-induced-minor-free graphs has bounded tree-independence number. 
Our result shows that this conjecture is true when~$H$ is the disjoint union of cycles of length~$\ell$ for every fixed integer ${\ell\geq3}$.

\begin{restatable}{corollary}{tinbound}\label{cor:K1tfree}
    Let~$k$, $\ell$, and~$t$ be positive integers and let~$G$ be a $K_{1,t}$-free graph.
    Then in~$\abs{V(G)}^{\mathcal{O}(\ell)}$ time, one can either find an induced packing of~$k$ cycles of length at least~$\ell$ in~$G$, or construct a tree-decomposition of~$G$ with independence number at most~${\mathcal{O}(t\ell k\log k)}$. 
\end{restatable}

We remark that for integers~${t \geq 3}$ and~${k \geq 2}$, the class of $K_{1,t}$-free graphs has unbounded tree-independence number, and also the class of graphs containing no induced packing of~$k$ cycles of length at least~$\ell$ has unbounded tree-independence number, see~\cite[Theorem~7.3]{DallardMS2024}. 
Thus, both conditions are necessary to bound the tree-independence number in Corollary~\ref{cor:K1tfree}. 

Bonamy et al.~\cite{Bonamy2024} constructed a class of $K_{2,3}$-free graphs without an induced packing of two cycles which has unbounded tree-independence number.
This implies that the $K_{1,t}$-free condition in Corollary~\ref{cor:K1tfree} cannot be replaced by a $K_{s,t}$-free condition for $s\geq2$ and $t\geq3$.
Moreover, the same construction shows that for every graph~$F$ with two adjacent vertices of degree at least~$4$,  we cannot replace the $K_{1,t}$-free condition by an $F$-free condition.

Corollary~\ref{cor:K1tfree} implies that a number of NP-hard problems, such as \textsc{Maximum Weight Independent Set}, \textsc{Maximum Weight Induced Matching}, and \textsc{Graph Homomorphism}, can be solved in polynomial time on every class of~$K_{1,t}$-free graphs containing no induced packing of~$k$ cycles of length at least~$\ell$~\cite{DBLP:conf/soda/Yolov18,DallardMS2024}. 
Bonamy et al.~{\cite[Conjecture~1.3]{Bonamy2024}} conjectured that \textsc{Maximum Independent Set} can be solved in polynomial time on graphs having no induced packing of~$k$ cycles, and our result shows that this conjecture holds for~$K_{1,t}$-free graphs. They proved that \textsc{Maximum Independent Set} can be solved in quasi-polynomial time on graphs having no induced packing of~$k$ cycles.
Lima et al.~\cite{LimaMMIRS2024} obtained an algorithmic metatheorem for the problem of finding a maximum-weight sparse  induced subgraph satisfying a fixed property expressible in counting monadic second-order logic (CMSO$_2$). 
This metatheorem implies that the \textsc{Minimum Weight Feedback Vertex Set} and \textsc{Minimum Weight Odd Cycle Transversal} can be solved in polynomial time on every class of $K_{1,t}$-free graphs containing no induced packing of~$k$ cycles. 
All these problems are known to be NP-hard on $K_{1,t}$-free graphs for some small constant~$t$. For example, \textsc{Maximum Independent Set} is NP-hard on planar graphs of maximum degree~$3$~\cite{GareyJS1976}, which are~${K_{1,4}}$-free. 

Korhonen~\cite{Korhonen2023} proved the grid induced minor theorem which shows that for every class~$\mathcal{C}$ of graphs of bounded maximum degree, if a graph~${G \in \mathcal{C}}$ has sufficiently large tree-width, then it contains a large grid as an induced minor. 
This can be used to show that if~${G \in \mathcal{C}}$ has no large induced packing of cycles, then~$G$ has a small feedback vertex set, whose size depends on the grid induced minor theorem. 
\Cref{thm:main1} with $\ell=3$ gives a tighter bound. 

\begin{corollary}\label{cor:boundeddegree}
    Let~$k$ and~$t$ be positive integers and let~$G$ be a graph of maximum degree at most~$t$.
    Then one can find in polynomial time either an induced packing of~$k$ cycles or a set~$X$ of at most~${\mathcal{O}(tk \log k)}$ vertices such that~${G-X}$ is a forest. 
\end{corollary}

We remark that Bonamy et al.~\cite{Bonamy2024} recently proved that for every graph~$G$ having no~$K_{t,t}$ as a subgraph, if~$G$ has no large induced packing of cycles, then the feedback vertex set number of~$G$ is at most logarithmic in the number of vertices of~$G$. 

On restricted classes of graphs, we prove a slightly stronger statement than Theorem~\ref{thm:main1}.

\begin{theorem}\label{thm:main2}
    Let~$\mathcal{C}$ be a class of graphs closed under taking subgraphs and let~$h$ be a function such that for every positive integer~$k$, every graph~${G \in \mathcal{C}}$ contains either~$k$ vertex-disjoint cycles or a set~$X$ of at most~$h(k)$ vertices such that~${G-X}$ is a forest.
    Then for every graph~${G \in \mathcal{C}}$ and an integer $\ell\geq3$, one can find in~$\abs{V(G)}^{\mathcal{O}(\ell)}$ time either an induced packing of~$k$ cycles of length at least~$\ell$ or sets~$X_1$ and~$X_2$ of vertices with
    \begin{align*}
        \abs{X_1}&<(28\ell-21)h(k)+(74\ell-54)k-78\ell+57,\\
        \abs{X_2}&<7h(k)+20k-21
    \end{align*}
    such that neither~${G-B_G(X_1,1)}$ nor ${G-B_G(X_2,\ell)}$ has a cycle of length at least~$\ell$.
\end{theorem}

For planar graphs, we obtain a tighter bound when $\ell=3$.

\begin{restatable}{theorem}{planar}\label{thm:main3}
    Let $k$ be a positive integer and let~$G$ be a planar multigraph.
    Then one can find in polynomial time either an induced packing of~$k$ cycles or a set~$X$ of at most~$5k$ vertices such that~${G-B_G(X,1)}$ is a forest. 
\end{restatable}

We also consider an analogue of Theorem~\ref{thm:main1} for distance-$d$ packings of cycles.
For a positive integer~$d$, a set~$\mathcal{F}$ of cycles in a graph~$G$ is a \emph{distance-$d$ packing} in~$G$ if~$G$ has no path of length at most~$d$ between any two distinct cycles in~$\mathcal{F}$. 
Note that a collection of cycles in~$G$ is an induced packing if and only if it is a distance-$1$ packing. 

We prove that if a graph has no distance-$d$ packing of two cycles, then it has a small set of bounded radius balls hitting all cycles.  

\begin{restatable}{theorem}{distanced}
    \label{thm:main4}
    Let~$d$ be a positive integer and let~$G$ be a graph.
    Then one can find in polynomial time either a distance-$d$ packing of two cycles or sets~$X_1$ and~$X_2$ of vertices with~${\abs{X_1} \leq 12(d+1)}$ and~${\abs{X_2} \leq 12}$ such that both~${G-B_G(X_1,2d)}$ and~${G-B_G(X_2,3d)}$ are forests. 
\end{restatable}

Theorem~\ref{thm:main4} has the following corollary on graphs of bounded maximum degree.

\begin{restatable}{corollary}{tinboundd}\label{cor:K1tfreed}
    Let~$d$ and~$t$ be positive integers and let~$G$ be a graph of maximum degree at most~$t$.
    Then one can find in polynomial time either a distance-$d$ packing of two cycles or a set~$X$ of at most~$\mathcal{O}(d\cdot t^{2d})$ vertices such that~${G-X}$ is a forest. 
\end{restatable}

Compared to Corollary~\ref{cor:K1tfree}, we show that the class of $K_{1,3}$-free graphs without a distance\nobreakdash-$2$ packing of two cycles has unbounded tree-independence number. This is certified by the line graphs of complete bipartite graphs, where Dallard et al.~\cite{DallardKKMMW2024} proved that they have large tree-independence number. 

\begin{restatable}{theorem}{unboundedtin}\label{thm:unboundedtin}
    The class of $K_{1,3}$-free graphs without a distance-$2$ packing of two cycles has unbounded tree-independence number. 
\end{restatable}

We conjecture that the following common generalisation of~Theorems~\ref{thm:main1} and~\ref{thm:main4} holds. 

\begin{conjecture}
    There exist functions~${f(k) = \mathcal{O}(k\log k)}$ and~${g(d) = \mathcal{O}(d)}$ such that for all positive integers~$k$ and~$d$, every graph~$G$ contains either a distance-$d$ packing of~$k$ cycles or a set~$X$ of at most $f(k)$ vertices such that~${G-B_G(X,g(d))}$ is a forest.
\end{conjecture}

We remark that Dujmovi\'{c} et al.~\cite{DujmovicJMM} recently showed that the conjecture holds for functions $f(k)=\mathcal{O}(k^{18}\log^{18}k)$ and $g(d)=19d$.

We now sketch the proofs of Theorems~\ref{thm:main1} and~\ref{thm:main4}.
The main idea for both proofs is to recursively find special ear-decompositions in $G$ where the attachment vertices of each ear are sufficiently far apart from the attachment vertices of all the previous ears.
The precise priority rules we use to construct these ear-decompositions are slightly different in the proofs of Theorems~\ref{thm:main1} and~\ref{thm:main4}.
However, in both cases, either we find the required packing of cycles inside the ear-decomposition, or we find a small set of bounded radius balls hitting all cycles (possibly with length constraint) in~$G$.

Our paper is organized as follows. 
In Section~\ref{sec:prelim}, we introduce preliminary concepts. 
In Section~\ref{sec:coarse}, we introduce our main tool called $\ell$-coarse ear-decompositions, and derive their properties. 
With these properties, in Section~\ref{sec:maintheorem}, we prove Theorem~\ref{thm:main1} as well as Theorem~\ref{thm:main2}. 
In Section~\ref{sec:planar}, we prove Theorem~\ref{thm:main3} and in Section~\ref{sec:distancepacking}, we prove Theorem~\ref{thm:main4}. 
We prove Corollaries~\ref{cor:K1tfree} and~\ref{cor:K1tfreed} and Theorem~\ref{thm:unboundedtin} in Section~\ref{sec:tin}. 
We discuss some open problems in Section~\ref{sec:conclusion}.

\section{Preliminaries}
\label{sec:prelim}

For an integer~$i$, we denote by~$[i]$ the set of positive integers at most~$i$. 
Note that if~${i \leq 0}$, then~$[i]$ is empty. 
Later, we consider the lexicographic order of pairs of non-negative integers.
For pairs~$(i,j)$ and~$(i',j')$ of non-negative integers, we write $(i,j)<_L(i',j')$ if~$(i,j)$ is lexicographically smaller than~$(i',j')$, and $(i,j)\leq_L(i',j')$ if $(i,j)<_L(i',j')$ or $(i,j)=(i',j')$.
The relations~$>_L$ and~$\geq_L$ are defined similarly.

Let~$G$ be a graph. 
We denote by~${V(G)}$ and~${E(G)}$ the vertex set and the edge set of~$G$, respectively. 
For a set~$A$ of vertices in~$G$, we denote by~${G-A}$ the graph obtained from~$G$ by deleting all the vertices in~$A$ and all edges incident with vertices in~$A$. 
In addition, we denote by~${G[A]}$ the subgraph of~$G$ induced by~$A$. 
If~${A = \{v\}}$, then we write~${G-v}$ for~${G-A}$. 
For a set~$F$ of edges in~$G$, we denote by~${G-F}$ the graph obtained by deleting all edges in~$F$. 
For another graph~$G'$, let~${G \cup G' \coloneqq (V(G)\cup V(G'),E(G)\cup E(G'))}$ and~${G \cap G' \coloneqq (V(G)\cap V(G'),E(G)\cap E(G'))}$.

For a path~$P$ of~$G$ and vertices~${a,b\in V(P)}$, we denote by~$aPb$ the subpath of~$P$ between~$a$ and~$b$. 
For sets~${X,Y \subseteq V(G)}$, an \emph{$(X,Y)$-path} is a path in~$G$ such that one of its ends is in~$X$, the other end is in~$Y$, and no internal vertex is in~$X\cup Y$.
If~${X = \{x\}}$ or~${Y = \{y\}}$, then we may call it an ${(x,Y)}$-path or an ${(X,y)}$-path, respectively.
For subgraphs~$H$ and~$H'$ of~$G$, an \emph{$(H,H')$-path} is a $(V(H),V(H'))$-path in~$G$ of length at least~$1$ whose edge set is disjoint from $E(H\cup H')$.
When $H=H'$, we call it an \emph{$H$-path}.
A \emph{chord} of~$H$ in~$G$ is an edge of~$G$ which forms an $H$-path of length~$1$.

An \emph{independent set} in~$G$ is a set of pairwise non-adjacent vertices.
The \emph{independence number} of~$G$, denoted by~${\alpha(G)}$, is the maximum size of an independent set in~$G$. 
The \emph{girth} of~$G$ is the length of a shortest cycle of~$G$. 

For vertices~$u$ and~$v$ of~$G$, the \emph{distance between~$u$ and~$v$ in~$G$}, denoted by~$\dist_G(u,v)$, is the length of a shortest path of~$G$ between~$u$ and~$v$; 
if no such path exists, then we set~${\dist_G(u,v) \coloneqq \infty}$. 
For a set~${X \subseteq V(G)}$, let~${\dist_G(u,X) \coloneqq \min_{v\in X}\dist_G(u,v)}$.
For a non-negative integer ${d}$, the \emph{ball of radius~$d$ around a vertex~$v$} in~$G$ is
\[
    B_G(v,d) \coloneqq \{ u \in V(G) \colon \dist_G(u,v) \leq d\}. 
\]
For a set~${S\subseteq V(G)}$, we write~${B_G(S,d)}$ for~${\bigcup_{v\in S}B_G(v,d)}$, 
and for a subgraph~$H$ of~$G$, we write~${B_G(H,d)}$ for~${B_G(V(H),d)}$. 

A graph is \emph{cubic} if every vertex has degree~$3$. 
The following lemma shows that every $n$-vertex cubic graph has no feedback vertex set of size at most~${\lfloor(n+1)/7\rfloor}$. 

\begin{lemma}\label{lem:cubic}
    Let~$k$ be a positive integer and let~$G$ be a cubic graph.
    If~${\abs{V(G)} \geq 7k-1}$, then~$G$ has no set~$X$ of at most~$k$ vertices such that~${G-X}$ is a forest. 
\end{lemma}

\begin{proof}
    Towards a contradiction, suppose that~$G$ has a set~$X$ of at most~$k$ vertices such that~${G-X}$ is a forest.
    For each~${i \in [3] \cup \{0\}}$, let~$n_i$ be the number of degree-$i$ vertices of~${G-X}$. 
    Since~$G$ is cubic and~${\abs{X} \leq k}$, we have that~${n_0 + n_1 + n_2 \leq 3k}$. 
    This implies that ${n_3 \geq 3k-1}$ as ${\abs{V(G-X)} \geq 6k-1}$.
    Since~${G-X}$ is a forest, by the Handshaking Lemma, 
    \[
        n_1+2n_2+3n_3\leq2(n_1+n_2+n_3-1),
    \]
    and therefore~${n_3 \leq n_1-2 \leq 3k-2}$, a contradiction. 
\end{proof}

We will use the following result of Simonovits~\cite{Simonovits1967}.
We define~$s_k$ for~${k \in \mathbb{N}}$ as 
\begin{align*}
    s_k\coloneqq
    \begin{cases}
        4k(\log k+\log\log k+4)\quad &\text{if } k \geq 2\\
        2 &\text{if } k = 1.
    \end{cases}
\end{align*}

\begin{theorem}[Simonovits~\cite{Simonovits1967}]
    \label{thm:simonovitz}
    Let~$k$ be a positive integer and let~$G$ be a cubic graph.
    If~${\abs{V(G)} \geq s_k}$, then~$G$ contains~$k$ vertex-disjoint cycles. 
    Moreover, these~$k$ cycles can be found in polynomial time.
\end{theorem}

\Cref{thm:simonovitz} implies that every cubic graph on at least~$40$ vertices has two vertex-disjoint cycles. 
We can improve this as follows.

\begin{lemma}\label{lem:twocycles}
    Let~$G$ be a cubic graph on at least eight vertices. 
    Then~$G$ has two vertex-disjoint cycles. 
\end{lemma}

\begin{proof}
    Let~${n \coloneqq \abs{V(G)}}$, let~$C$ be a shortest cycle of~$G$, and let~${G' \coloneqq G-V(C)}$. 
    It suffices to show that~$G'$ has average degree at least~$2$, and hence contains a cycle.
    Since~$G$ is cubic and~$C$ is an induced cycle, we deduce that~$G'$ has at least one vertex.

    Towards a contradiction, suppose that the average degree of~$G'$ is less than~$2$.
    Then
    \[
        \abs{V(G')}=n-\abs{V(C)}>\abs{E(G')}\geq3\abs{V(G')}-\abs{V(C)}=3n-4\abs{V(C)},
    \]
    and therefore $\abs{V(C)}\geq2n/3>5$.
    Since~$C$ is a shortest cycle and has length at least~$6$, every vertex in $V(G)\setminus V(C)$ has at most one neighbour in~$V(C)$.
    Hence,~$G'$ has minimum degree at least~$2$, a contradiction.
\end{proof} 

We remark Lemma~\ref{lem:twocycles} is tight since there is no cubic graph on seven vertices, and the complete bipartite graph~$K_{3,3}$ is cubic but does not have two vertex-disjoint cycles.

For a graph~$G$, a \emph{tree-decomposition} of~$G$ is a pair~${\mathcal{T} = (T,\beta)}$ of a tree~$T$ and a function~${\beta \colon V(T)\to 2^{V(G)}}$ whose images are called the \emph{bags} of~$\mathcal{T}$ such that 
\begin{enumerate}
    \item ${\bigcup_{u \in V(T)} \beta(u) = V(G)}$,
    \item for every~${uv \in E(G)}$ there exists some~${t \in V(T)}$ with~${\{u,v\} \subseteq \beta(t)}$, and 
    \item for every vertex~${v \in V(G)}$ the set~${\{ u \in V(T) \colon v \in \beta(u) \}}$ induces a subtree of~$T$. 
\end{enumerate}
The \emph{width} of~${(T,\beta)}$ is~${\max\{ \abs{\beta(u)}-1 \colon u \in V(T) \}}$ and the \emph{tree-width} of~$G$ is the minimum width among all possible tree-decompositions of~$G$. 
It is well known that every forest has tree-width at most~$1$. 

Birmele~\cite{Birmele03} showed that graphs without long cycles have small tree-width.

\begin{lemma}[Birmele~\cite{Birmele03}]\label{lem:birmele}
    Let~$G$ be a graph and~$\ell$ be an integer at least~$3$.
    If~$G$ has no cycle of length at least~$\ell$, then one can construct a tree-decomposition of~$G$ of width at most~$\ell-2$ in linear time.
\end{lemma}

The \emph{independence number} of a tree-decomposition~$\mathcal{T}$, denoted by~$\alpha(\mathcal{T})$, is defined as
\[
    \alpha(\mathcal{T}) \coloneqq \max_{t\in V(T)} \alpha(G[\beta(t)]). 
\]
The \emph{tree-independence number} of a graph~$G$, denoted by~$\tin(G)$, is the minimum independence number among all possible tree-decompositions of~$G$. 

We will use the following algorithm to find a shortest path or cycle among all path or cycle of length at least certain number.

\begin{lemma}\label{lem:shortest}
    Let~$d$ be a positive integer, let~$G$ be a graph, and let~$s$ and~$t$ be distinct vertices of~$G$.
    Among the paths of length at least~$d$ between~$s$ and~$t$ and the cycles of length at least~$d$ passing through~$s$, one can find shortest ones in~$\mathcal{O}(n^{d-1}(n+m)+nm)$ time.
\end{lemma}
\begin{proof}
    If $d=1$, then the statement holds by finding a shortest path between~$s$ and~$t$ in~$\mathcal{O}(n+m)$ time and a shortest cycle passing through~$s$ in~$\mathcal{O}(nm)$ time.
    Thus, we may assume that~${d\geq2}$.

    We first find a shortest path of length at least~$d$ between~$s$ and~$t$.
    We may assume that~${d\geq2}$, because otherwise it can be done in~$\mathcal{O}(n+m)$ time by finding a shortest path between~$s$ and~$t$.
    We enumerate all paths $sv_1\cdots v_{d-1}$ in~$G-t$ of length~${d-1}$.
    This can be done in $\mathcal{O}(n^{d-1}d)$ time as there are at most $n^{d-1}$ choices of the vertices $v_1,\ldots,v_{d-1}$, and for each choice, we can check whether $sv_1\cdots v_{d-1}$ is a path in $\mathcal{O}(d)$ time.
    For each path $sv_1\cdots v_{d-1}$, we find a shortest path of $G-(\{v_j:j\in[d-2]\}\cup\{s\})$ between~$v_{d-1}$ and~$t$ in~$\mathcal{O}(n+m)$, if it exists.
    By concatenating $sv_1\cdots v_{d-1}$ and the shortest path, we gain a path of length at least~$d$ between~$s$ and~$t$.
    In addition, at least one shortest path of length at least~$d$ between~$s$ and~$t$ is found by the process, if it exists.
    Thus, one can find a shortest path of length at least~$d$ between~$s$ and~$t$ in $\mathcal{O}(n^{d-1}(n+m))$ time.
    
    We now find a shortest cycle of length at least~$d$ passing through~$s$.
    We may assume that $d\geq3$, because otherwise it can be done in~$\mathcal{O}(nm)$ time by finding a shortest cycle passing through~$s$.
    For every pair $(s',t')$ of distinct neighbours of~$s$, we find a shortest path of~$G-s$ of length at least~$d-2$ between them, if it exists.
    Note that this can be done in~$\mathcal{O}(n^{d-1}(n+m))$.
    By concatenating~$s'st'$ and the shortest path, we gain a cycle of length at least~$d$ passing through~$s$.
    In addition, at least one shortest cycle of length at least~$d$ passing through~$s$ is found by the process, if it exists.
    Thus, one can find a shortest cycle of length at least~$d$ passing through~$s$ in $\mathcal{O}(n^{d-1}(n+m))$ time.
\end{proof}

\section{Coarse ear-decompositions}
\label{sec:coarse}

Before proving Theorem~\ref{thm:main1}, in this section, we introduce $\ell$-coarse ear-decompositions in graphs and investigate their properties.
For a graph~$G$, we denote by~$V_{\geq 3}(G)$ the set of vertices of~$G$ having degree at least~$3$.
For positive integers~${\ell,t,\mu_1,\ldots,\mu_t}$ with $\ell\geq3$, an \emph{$\ell$-coarse ear-decomposition} in~$G$ is a subgraph
\[
    \mathcal{H} \coloneqq \bigcup_{i\in[t]} \bigcup_{j\in[\mu_i]} P_{i,j}
\]
of~$G$ where~$\{P_{i,j} : i \in [t], j \in [\mu_i]\}$ is a collection of pairwise edge-disjoint paths and cycles in~$G$ satisfying conditions {\ref{cond:coarse1}--\ref{cond:coarse3}} below.
Each~$P_{i,1}$ is a cycle which shares at most one vertex with the union of the previous~$P_{i',j'}$ (in the lexicographic order), and each $P_{i,j}$ with ${j>1}$ is a path sharing only its ends with the previous $P_{i',j'}$.
We denote by~$\mathcal{P}(\mathcal{H})$ the set of all pairs $(i,j)$ with $i\in[t]$ and $j\in[\mu_i]$.

The formal definition of an $\ell$-coarse ear-decomposition is as follows.
Let~$H_{0,0}$ be a null graph, let~$\mu_0 \coloneqq0$, and let~$Y_{0,0}$ and~$Z_{0,0}$ be empty sets.
For each $(i,j)\in\mathcal{P}(\mathcal{H})$,
\[
    H_{i,j} \coloneqq \bigcup_{\substack{(i',j')\in\mathcal{P}(\mathcal{H})\\(i',j')\leq_L(i,j)}}P_{i',j'},
\]
let ${Y_{i,j}\coloneqq B_{H_{i,j}}(V_{\geq3}(H_{i,j}),\ell-1)}$, and let ${Z_{i,j}\coloneqq B_{G-(V(H_{i,j})\setminus Y_{i,j})}(Y_{i,j},1)}$. 

An \emph{$\ell$-extendable path} of~$H_{i,j}$ is an $H_{i,j}$-path~$P$ in~$G-Z_{i,j}$ with ends~$a$ and~$b$ such that either~$a$ and~$b$ are in distinct components of~$H_{i,j}$, or it holds that
\[
    \abs{E(P)}+\dist_{H_{i,j}}(a,b)\geq\ell.
\]
An \emph{$\ell$-cycle} is a cycle of length at least~$\ell$.
Then the aforementioned conditions are as follows.

\begin{enumerate}[label=(\Alph*)]
    \item\label{cond:coarse1} For each $i\in[t]$, $P_{i,1}$ satisfies the following.
    \begin{enumerate}[label=(A\arabic*)]
        \item\label{cond:coarse1-1} If $G-Z_{i-1,\mu_{i-1}}$ has an $\ell$-cycle intersecting~$V(H_{i-1,\mu_{i-1}})$ at exactly one vertex, then~$P_{i,1}$ is a shortest such cycle with intersecting vertex~$c_i$.
        We say that~$P_{i,1}$ is of \emph{type~1}.
        \item\label{cond:coarse1-2} Otherwise, $P_{i,1}$ is a shortest $\ell$-cycle of~$G-Z_{i-1,\mu_{i-1}}$.
        We say that~$P_{i,1}$ is of \emph{type~2}.
    \end{enumerate}
    \item\label{cond:coarse2} For each~$i\in[t]$ and~$j\in[\mu_i]\setminus\{1\}$, $P_{i,j}$ is a shortest $\ell$-extendable path of~$H_{i,j-1}$ with ends~$a_{i,j}$ and~$b_{i,j}$.
    \item\label{cond:coarse3} For each $i\in[t]$, there is no $\ell$-extendable path of~$H_{i-1,\mu_{i-1}}$.
\end{enumerate}
If~$G$ has no $\ell$-cycle, then we define~$\mathcal{H}$ as a null graph with empty $\mathcal{P}(\mathcal{H})$.

For each $(i,j)\in\mathcal{P}(\mathcal{H})$, let
\[
    \Gamma_{i,j} \coloneqq
    \begin{cases}
        V(P_{i,j})\cap V(H_{i-1,\mu_{i-1}}) &\text{if }j=1,\\
        V(P_{i,j})\cap V(H_{i,j-1}) &\text{otherwise.}
    \end{cases}
\]

Throughout this paper, $\ell$ is a fixed integer at least~$3$, and when we are given an $\ell$-coarse ear-decomposition~$\mathcal{H}$, we will use those notations $H_{i,j}$, $Y_{i,j}$, $Z_{i,j}$, $\Gamma_{i,j}$, $a_{i,j}$, $b_{i,j}$, and~$c_i$ without repeatedly defining them.
Note that each~$H_{i,j}$ is an $\ell$-coarse ear-decomposition in~$G$, and for all $(i,j),(i',j')\in\mathcal{P}(\mathcal{H})$ with $(i,j)\leq_L(i',j')$, we have that ${Y_{i,j}\subseteq Y_{i',j'}}$ and ${Z_{i,j}\subseteq Z_{i',j'}}$.

We say that~$\mathcal{H}$ is \emph{maximal} if $G-Z_{t,\mu_t}$ has neither an $\ell$-extendable path of~$H_{t,\mu_t}$ nor an $\ell$-cycle intersecting~$V(H_{t,\mu_t})$ at less than two vertices.
The \emph{branch vertices} of~$\mathcal{H}$ are the vertices in~$V_{\geq3}(\mathcal{H})$, and the \emph{admissible vertices} of~$\mathcal{H}$ are the vertices in $V(\mathcal{H})\setminus Y_{t,\mu_t}$.
Note that the maximum degree of~$\mathcal{H}$ is at most~$4$ and every admissible vertex of~$\mathcal{H}$ has degree~$2$ in~$\mathcal{H}$.
We remark that each of~$a_{i,j}$ and~$b_{i,j}$ is an admissible vertex of~$H_{i,j-1}$ and each~$c_i$ is an admissible vertex of~$H_{i-1,\mu_{i-1}}$.
In addition, for every admissible vertex $v$ of~$\mathcal{H}$, there exists a unique $(i,j)\in\mathcal{P}(\mathcal{H})$ such that~$v \in V(P_{i,j})$.

We now analyse the running time of constructing a maximal $\ell$-coarse ear-decomposition.

\begin{lemma}
    For a graph~$G$ with~$n$ vertices and~$m$ edges, one can construct a maximal $\ell$-coarse ear-decomposition in~$G$ in~${\mathcal{O}(n^{\ell+1}(n+m))}$ time.
\end{lemma}
\begin{proof}
    We first show that for every $\ell$-coarse ear-decomposition~$\mathcal{H}$ in~$G$, we have $\abs{\mathcal{P}(\mathcal{H})}\leq2n$.
    On the one hand, the number of $(i,j)\in\mathcal{P}(\mathcal{H})$ with $\abs{V(P_{i,j})}\geq3$ is at most $n-\ell+1$ because~$P_{i,j}$ contains a vertex~$v$ for which~$(i,j)$ is the unique pair in~$\mathcal{P}(H_{i,j})$ such that $v\in V(P_{i,j})$.
    On the other hand, the number of $(i,j)\in\mathcal{P}(\mathcal{H})$ with $\abs{V(P_{i,j})}=2$ is at most~$\lfloor n/2\rfloor$ as each of~$a_{i,j}$ and~$b_{i,j}$ has degree~$2$ in~$H_{i,j-1}$.
    Thus, $\abs{\mathcal{P}(\mathcal{H})}\leq2n$.

    Hence, it suffices to show that given an $\ell$-coarse ear-decomposition $\mathcal{H}\coloneqq\bigcup_{i\in[t]}\bigcup_{j\in[\mu_i]}P_{i,j}$ in~$G$, one can either construct an $\ell$-coarse ear-decomposition in~$G$ by attaching one more path or cycle, or decide that $\mathcal{H}$ is maximal in $\mathcal{O}(n^\ell(n+m))$ time.
    Let $G'\coloneqq G-Z_{t,\mu_t}$.

    We first find an $\ell$-extendable path of~$\mathcal{H}$.
    For each pair~$(s,t)$ of distinct admissible vertices of~$\mathcal{H}$, by Lemma~\ref{lem:shortest} for $d=\max\{\ell-\dist_\mathcal{H}(s,t),1\}$ and ${G'-(V(\mathcal{H})\setminus\{s,t\})}$, one can find an $\ell$-extendable path of~$\mathcal{H}$ between~$s$ and~$t$ in ${\mathcal{O}(n^{d-1}(n+m)+nm)}$, if it exists.
    Since ${d\leq\ell-1}$, one can find a shortest $\ell$-extendable path of~$\mathcal{H}$ in ${\mathcal{O}(n^\ell(n+m))}$ time, if it exists.
    Thus, we may assume that there is no $\ell$-extendable path of~$\mathcal{H}$.
    
    We now find an $\ell$-cycle of~$G'$ intersecting~$V(\mathcal{H})$ at exactly one vertex.
    For each admissible vertex~$s$ of~$\mathcal{H}$, by Lemma~\ref{lem:shortest} for $d=\ell$ and ${G'-(V(\mathcal{H})\setminus\{s\})}$, one can find an $\ell$-cycle of~$G'$ intersecting~$V(\mathcal{H})$ only at~$s$ in~${\mathcal{O}(n^{\ell-1}(n+m))}$ time, if it exists.
    Hence, one can find a shortest $\ell$-cycle of~${G-Z_{t,\mu_t}}$ intersecting~$V(\mathcal{H})$ at exactly one vertex in ${\mathcal{O}(n^\ell(n+m))}$ time, if it exists.
    Thus, we may assume that there is no such cycle.

    Finally, we find an $\ell$-cycle of $G'-V(\mathcal{H})$ in ${\mathcal{O}(n^\ell(n+m))}$ time, if it exists, by applying Lemma~\ref{lem:shortest} to each vertex of ${G'-V(\mathcal{H})}$ for ${d=\ell}$ and ${G'-V(\mathcal{H})}$.
    If ${G'-V(\mathcal{H})}$ has no $\ell$-cycle, then we correctly decide that~$\mathcal{H}$ is maximal.
    This completes the proof.
\end{proof}

From the definition of an $\ell$-coarse ear-decomposition, we observe the following.

\begin{observation}\label{obs:structure}
    Let $\mathcal{H}\coloneqq\bigcup_{i\in[t]}\bigcup_{j\in[\mu_i]}P_{i,j}$ be an $\ell$-coarse ear-decomposition in a graph~$G$.
    Then the following hold.
    \begin{enumerate}[label=(\alph*)]
        \item\label{item:type2} For each $i\in[t]$, if $P_{i,1}$ is of type~2 and has a chord in~$G$, then $\abs{E(P_{i,1})}\leq2(\ell-1)$.
        \item\label{item:admissible} For $(i,j),(i',j')\in\mathcal{P}(\mathcal{H})$ with $(i,j)\leq_L(i',j')$, if a vertex $v\in V(H_{i,j})$ is an admissible vertex of~$H_{i',j'}$, then it is also an admissible vertex of~$H_{i,j}$.
        \item\label{item:extendable} For an $\ell$-extendable path~$Q$ of~$\mathcal{H}$, if both ends of~$Q$ are in $V(H_{i,j})$ for some $(i,j)\in\mathcal{P}(\mathcal{H})$, then~$Q$ is an $\ell$-extendable path of~$H_{i,j}$.
    \end{enumerate}
\end{observation}

We show that every cycle in an $\ell$-coarse ear-decomposition is indeed an $\ell$-cycle.

\begin{lemma}\label{lem:ell cycle}
    Let $\mathcal{H}\coloneqq\bigcup_{i\in[t]}\bigcup_{j\in[\mu_i]}P_{i,j}$ be an $\ell$-coarse ear-decomposition in a graph~$G$.
    Then every cycle in~$\mathcal{H}$ is an $\ell$-cycle in~$G$.
\end{lemma}
\begin{proof}
    We prove by induction on $x\coloneqq\abs{\mathcal{P}(\mathcal{H})}$.
    The statement obviously holds for~$x=0$.
    Thus, we may assume that $x\geq1$.
    Since $P_{t,1}$ is an $\ell$-cycle, if $\mu_t=1$, then the statement easily follows from the inductive hypothesis as $P_{t,1}$ is the unique cycle in~$\mathcal{H}$ that is not a cycle in~$H_{t-1,\mu_{t-1}}$.
    Thus, we may assume that $\mu_t\geq2$.
    
    Let~$C$ be an arbitrary cycle in~$\mathcal{H}$.
    We show that~$C$ is an $\ell$-cycle in~$G$.
    By the inductive hypothesis, we may assume that~$C$ is not a cycle in~$H_{t,\mu_t-1}$.
    Since every internal vertex of~$P_{t,\mu_t}$ has degree~$2$ in~$\mathcal{H}$, every edge of~$P_{t,\mu_t}$ is an edge of~$C$.
    By the definition of an $\ell$-extendable path, we have that
    \[
    	\abs{E(C)}\geq\abs{E(P_{t,\mu_t})}+\dist_{H_{t,\mu_t-1}}(a_{t,\mu_t},b_{t,\mu_t})\geq\ell.
    \]
    Thus,~$C$ is an $\ell$-cycle in~$G$.
    
    This completes the proof by induction.
\end{proof}

We present equivalent conditions for having degree larger than~$2$ in an $\ell$-coarse ear-decomposition.

\begin{lemma}\label{lem:unique}
    Let $\mathcal{H}\coloneqq\bigcup_{i\in[t]}\bigcup_{j\in[\mu_i]}P_{i,j}$ be an $\ell$-coarse ear-decomposition in a graph and let~$v$ be a vertex of~$\mathcal{H}$.
    Then the following hold.
    \begin{itemize}
        \item $\deg_\mathcal{H}(v)=3$ if and only if there exists a unique $(i,j)\in\mathcal{P}(\mathcal{H})$ such that ${v\in\{a_{i,j},b_{i,j}\}}$.
        \item $\deg_\mathcal{H}(v)=4$ if and only if there exists a unique ${i\in[t]}$ such that $v=c_i$.
    \end{itemize}
\end{lemma}
\begin{proof}
    We proceed to show by induction on~${x \coloneqq \abs{\mathcal{P}(\mathcal{H})}}$ that both statements hold. 
    The statements obviously hold for~${x = 0}$.
    Thus, we may assume that~${x \geq 1}$.

    Suppose that $\mu_t\geq2$.
    Since~$a_{t,\mu_t}$ and~$b_{t,\mu_t}$ are admissible vertices of~$H_{t,\mu_t-1}$, they have degree~$2$ in~$H_{t,\mu_t-1}$.
    By induction, for every $(i,j)\in\mathcal{P}(H_{t,\mu_t-1})$, neither~$a_{t,\mu_t}$ nor~$b_{t,\mu_t}$ is in~$\Gamma_{i,j}$.
    Since~$a_{t,\mu_t}$ and~$b_{t,\mu_t}$ have degree~$3$ in~$\mathcal{H}$ and every internal vertex of~$P_{t,\mu_t}$ has degree~$2$ in~$\mathcal{H}$, we have that $V_{\geq3}(\mathcal{H})\setminus V_{\geq3}(H_{t,\mu_t-1})=\{a_{t,\mu_t},b_{t,\mu_t}\}$.
    Hence, the statements hold.

    Thus, we may assume that $\mu_t=1$.
    Note that $t\geq2$ as $x\geq2$.
    Let $H'\coloneqq H_{t-1,\mu_{t-1}}$.
    We may assume that $P_{t,1}$ is of type~1, because otherwise the statements easily follow from the inductive hypothesis as~$P_{t,1}$ and~$H'$ are vertex-disjoint.
    Since~$c_t$ is an admissible vertex of~$H'$, it has degree~$2$ in~$H'$.
    By induction, for every $(i,j)\in\mathcal{P}(H')$, $c_t\notin\Gamma_{i,j}$.
    Since~$c_t$ has degree~$4$ in~$\mathcal{H}$ and every vertex in $V(P_{t,1})\setminus\{c_t\}$ has degree~$2$ in~$\mathcal{H}$, we have that $V_{\geq3}(\mathcal{H})\setminus V_{\geq3}(H')=\{c_t\}$.
    Hence, the statements hold.
    
    This completes the proof by induction.
\end{proof}

We have the following corollary of Lemma~\ref{lem:unique}.

\begin{corollary}\label{cor:close ends}
    Let $\mathcal{H}\coloneqq\bigcup_{i\in[t]}\bigcup_{j\in[\mu_i]}P_{i,j}$ be an $\ell$-coarse ear-decomposition in a graph.
    If~$\mathcal{H}$ has distinct branch vertices~$v_1$ and~$v_2$ with $\dist_\mathcal{H}(v_1,v_2)\leq\ell-1$, then there exists a unique $(p,q)\in\mathcal{P}(\mathcal{H})$ such that $\{v_1,v_2\}=\{a_{p,q},b_{p,q}\}$.
\end{corollary}
\begin{proof}
    For each ${i\in[2]}$, by Lemma~\ref{lem:unique}, there exists a unique $(p_i,q_i)\in\mathcal{P}(\mathcal{H})$ such that ${v_i\in\Gamma_{p_i,q_i}}$.
    Towards a contradiction, suppose that $(p_1,q_1)\neq(p_2,q_2)$.
    Without loss of generality, we may assume that $(p_1,q_1)<_L(p_2,q_2)$.
    Since $B_{H_{p_1,q_1}}(v_1,\ell-1)\subseteq Z_{p_1,q_1}$, we deduce that $B_{H_{p_1,q_1}}(v_1,\ell-1)=B_\mathcal{H}(v_1,\ell-1)$, and therefore $v_2\in B_{H_{p_1,q_1}}(v_1,\ell-1)\subseteq Z_{p_1,q_1}$, a contradiction.
    Thus, $(p_1,q_1)=(p_2,q_2)$.
    
    Since~$v_1$ and~$v_2$ are distinct, $P_{p_1,q_1}$ is not a cycle, so $\{v_1,v_2\}=\{a_{p_1,q_1},b_{p_1,q_1}\}$.
\end{proof}

From now on, we consider a graph~$G$ having no $\ell$-cycle of length at most~$2\ell$ and its $\ell$-coarse ear-decomposition~$\mathcal{H}$.
The following proposition shows that if~$\mathcal{H}$ is maximal and has a chord in~$G$, then both ends of the chord are close to some branch vertex of~$\mathcal{H}$.

\begin{proposition}\label{prop:induced}
    Let~$G$ be a graph which has no $\ell$-cycle of length at most~$2\ell$ and let ${\mathcal{H}\coloneqq\bigcup_{i\in[t]}\bigcup_{j\in[\mu_i]}P_{i,j}}$ be a maximal $\ell$-coarse ear-decomposition in~$G$.
    If~$\mathcal{H}$ has a chord~$ab$ in~$G$, then~$\mathcal{H}$ has a branch vertex~$x$ such that $\{a,b\}\subseteq B_\mathcal{H}(x,\ell-1)$.
\end{proposition}

To prove Proposition~\ref{prop:induced}, we will use the following three lemmas.
In the next lemma, we present properties of $\ell$-extendable paths of length at most~$\ell-1$.

\begin{lemma}\label{lem:lengthy}
    Let~$G$ be a graph which has no $\ell$-cycle of length at most~$2\ell$ and let ${\mathcal{H}\coloneqq\bigcup_{i\in[t]}\bigcup_{j\in[\mu_i]}P_{i,j}}$ be an $\ell$-coarse ear-decomposition in~$G$.
    If~$\mathcal{H}$ has an $\ell$-extendable path~$Q$ of length at most~$\ell-1$, then 
    \begin{itemize}
        \item $t\geq2$ and
        \item one end of~$Q$ is in~$V(H_{t-1,\mu_{t-1}})$ and the other is in~$V(P_{t,1})$.
    \end{itemize}
    In addition, $P_{t,1}$ is of type~2 and $\mu_t=1$.
\end{lemma}
\begin{proof}
    Let~$a$ and~$b$ be the ends of~$Q$.
    Note that~$a$ and~$b$ are admissible vertices of~$\mathcal{H}$.
    Thus, there are unique pairs $(p,q),(r,s)\in\mathcal{P}(\mathcal{H})$ such that $a\in V(P_{p,q})$ and $b\in V(P_{r,s})$.
    By symmetry, we may assume that $(p,q)\leq_L(r,s)$.
    We remark that every cycle in $\mathcal{H}\cup Q$ is an $\ell$-cycle in~$G$ by the definition of an $\ell$-extendable path.

    We consider the following four cases.

    \medskip
    \noindent\textbf{Case 1: $(p,q)=(r,s)$ and $P_{p,q}$ is a cycle of type~2.}

    We will derive a contradiction.
    By~\ref{cond:coarse1-2}, a shortest subpath of~$P_{p,q}$ between~$a$ and~$b$ is of length at most~$\abs{E(Q)}$.
    Thus, by concatenating the subpath and~$Q$, we obtain a cycle of length at most~$2\ell$ in~$\mathcal{H}$ which is an $\ell$-cycle in~$G$, a contradiction.

    \medskip
    \noindent\textbf{Case 2: $(p,q)=(r,s)$ and $P_{p,q}$ is not a cycle of type~2.}

    We will derive a contradiction.
    Note that $\Gamma_{p,q}\neq\emptyset$ so that $\widetilde{P_{p,q}}\coloneqq P_{p,q}-\Gamma_{p,q}$ is a path.
    Let~$P'_{p,q}$ be the graph obtained from~$P_{p,q}$ by replacing $a\widetilde{P_{p,q}}b$ with~$Q$.
    Note that~$a\widetilde{P_{p,q}}b$ and~$Q$ form a cycle in~$\mathcal{H}\cup Q$ which is an $\ell$-cycle in~$G$.
    Since~$G$ has no $\ell$-cycle of length at most~$2\ell$ and~$Q$ is of length at most~$\ell-1$, we deduce that~$a\widetilde{P_{p,q}}b$ is of length at least~$\ell+2$, and therefore~$P'_{p,q}$ is shorter than~$P_{p,q}$.
    Since~$a$ and~$b$ are admissible vertices of~$\mathcal{H}$, $\dist_{P_{p,q}}(\{a,b\},\Gamma_{p,q})\geq\ell$, and therefore the length of~$P'_{p,q}$ is at least $2\ell+1>\ell$.

    If $q=1$, then~$P'_{p,q}$ is an $\ell$-cycle of $G-Z_{p-1,\mu_{p-1}}$ intersecting $V(H_{p-1,\mu_{p-1}})$ only at~$c_p$, contradicting~\ref{cond:coarse1-1} as $P'_{p,q}$ is shorter than~$P_{p,q}$.
    Otherwise, $P'_{p,q}$ is an $\ell$-extendable path of~$H_{p,q-1}$, contradicting~\ref{cond:coarse2} by the same reason.

    \medskip
    \noindent\textbf{Case 3: $(p,q)<_L(r,s)$ and $s\geq2$.}

    We will derive a contradiction.
    Note that $a\notin\Gamma_{r,s}$ as it is an admissible vertex of~$\mathcal{H}$.
    Let~$Q'$ be the path obtained by concatenating~$Q$ and $bP_{r,s}b_{r,s}$.
    Since $a\in V(H_{r,s-1})$, by Observation~\ref{obs:structure}\ref{item:admissible}, $a$ is an admissible vertex of~$H_{r,s-1}$, and therefore~$Q'$ is an $H_{r,s-1}$-path in $G-Z_{r,s-1}$.
    Since~$b$ is an admissible vertex of~$H_{r,s}$, we have $\dist_{P_{r,s}}(b,\Gamma_{r,s})\geq\ell$.
    Thus, $\abs{E(P_{r,s})}>\abs{E(Q')}\geq\ell$, and therefore~$Q'$ is an $\ell$-extendable path of~$H_{r,s-1}$ shorter than~$P_{r,s}$, contradicting~\ref{cond:coarse2}.
    
    \medskip
    \noindent\textbf{Case 4: $(p,q)<_L(r,s)$ and $s=1$.}

    We now show that the statement holds.
    Since $(p,q)<_L(r,s)$, we have that $r\geq2$ and $a\in V(H_{r-1,\mu_{r-1}})$.
    In addition, every path of~$\mathcal{H}$ between~$a$ and~$b$ must contain a branch vertex of~$\mathcal{H}$.

    We first show that~$P_{r,1}$ is of type~2.
    Suppose not.
    Let~$P$ be the path obtained by concatenating~$Q$ and any subpath of~$P_{r,1}$ between~$b$ and~$c_r$.
    Since~$b$ is an admissible vertex of~$\mathcal{H}$, we have that
    \[
        \abs{E(P)}\geq1+\dist_{P_{r,1}}(b,c_r)\geq\ell+1.
    \]
    Recall that~$c_r$ is an admissible vertex of~$H_{r-1,\mu_{r-1}}$.
    Since $a\in V(H_{r-1,\mu_{r-1}})$, by Observation~\ref{obs:structure}\ref{item:admissible}, $a$ is an admissible vertex of~$H_{r-1,\mu_{r-1}}$.
    Therefore,~$P$ is an $\ell$-extendable path of~$H_{r-1,\mu_{r-1}}$, contradicting~\ref{cond:coarse3}.
    Hence, $P_{r,1}$ is of type~2.

    To show the statement, it suffices to show that $r=t$ and $\mu_t=1$.
    Towards a contradiction, suppose that~${r<t}$ or~${\mu_t\geq2}$.
    Since~$P_{r,1}$ is of type~2, $a$ and~$b$ are in distinct components of~$H_{r,1}$.
    Hence, $Q$ is an $\ell$-extendable path of~$H_{r,1}$.
    Since $r<t$ or $\mu_t\geq2$, we deduce from~\ref{cond:coarse3} that $\mu_r\geq2$.
    In addition, by~\ref{cond:coarse2}, for each $i\in[\mu_r]\setminus\{1\}$, the length of $P_{r,i}$ is at most $\abs{E(Q)}\leq\ell-1$.

    If $\Gamma_{r,2}\subseteq V(P_{r,1})$, then by~\ref{cond:coarse1-2}, a shortest subpath of~$P_{r,1}$ between~$a_{r,2}$ and~$b_{r,2}$ is of length at most~$\abs{E(P_{r,2})}$.
    Thus, by concatenating the subpath and~$P_{r,2}$, we obtain a cycle of length at most~$2\ell$ in~$\mathcal{H}$ which is an $\ell$-cycle in~$G$, a contradiction,
    
    Hence, at least one of~$a_{r,2}$ and~$b_{r,2}$ is in $V(H_{r-1,\mu_{r-1}})$.
    By symmetry, we may assume that $a_{r,2}\in V(H_{r-1,\mu_{r-1}})$.
    By~\ref{cond:coarse3} and Observation~\ref{obs:structure}\ref{item:extendable}, $b_{r,2}$ must be in $V(P_{r,1})$.
    Since~$a$ and~$b$ are admissible vertices of~$\mathcal{H}$, we have $\dist_\mathcal{H}(\{a,b\},\Gamma_{r,2})\geq\ell$.
    Note that no internal vertex of~$Q$ is in $V(\mathcal{H})$.
    Then by concatenating~$Q$, $P_{r,2}$, and any subpath of~$P_{r,1}$ between~$b$ and~$b_{r,2}$, we obtain an $\ell$-extendable path of~$H_{r-1,\mu_{r-1}}$, contradicting~\ref{cond:coarse3}.
    Hence, ${r=t}$ and ${\mu_t=1}$, so the statement holds.

    \medskip

    This completes the proof.
\end{proof}

In the following two lemmas, we consider chords of~$\mathcal{H}$ in~$G$.

\begin{lemma}\label{lem:no edge1}
    Let~$G$ be a graph which has no $\ell$-cycle of length at most~$2\ell$ and let ${\mathcal{H}\coloneqq\bigcup_{i\in[t]}\bigcup_{j\in[\mu_i]}P_{i,j}}$ be an $\ell$-coarse ear-decomposition in~$G$ such that ${\mu_t\geq2}$.
    If~$\mathcal{H}$ has a chord~$ab$ in~$G$ with ${a\in V(H_{t,\mu_t-1})\setminus\Gamma_{t,\mu_t}}$ and~${b\in V(P_{t,\mu_t})\setminus\Gamma_{t,\mu_t}}$, then there is $x\in\Gamma_{t,\mu_t}$ such that $\{a,b\}\subseteq B_\mathcal{H}(x,\ell-1)$.
\end{lemma}
\begin{proof}
    Note that~$a$ is an admissible vertex of~$H_{t,\mu_t-1}$, because otherwise $b\in Z_{t,\mu_t-1}$.
    Since~$b$ is an internal vertex of~$P_{t,\mu_t}$, the length of~$P_{t,\mu_t}$ is at least~$2$.

    We divide into two cases depending on the length of~$P_{t,\mu_t}$.

    \medskip
    \noindent\textbf{Case 1: $P_{t,\mu_t}$ is of length at least~$3$.}
    
    By symmetry, we may assume that $\dist_{P_{t,\mu_t}}(b,a_{t,\mu_t})\geq2$.
    Let~$P$ be the path obtained by concatenating~$ab$ and $bP_{t,\mu_t}b_{t,\mu_t}$.
    Note that~$P$ is shorter than~$P_{t,\mu_t}$.
    Since~$a$ is an admissible vertex of~$H_{t,\mu_t-1}$, $P$ is an $H_{t,\mu_t-1}$-path in~${G-Z_{t,\mu_t-1}}$.

    We show that $\{a,b\}\subseteq B_\mathcal{H}(b_{t,\mu_t},\ell-1)$.
    Suppose not.
    Then one of $\dist_\mathcal{H}(a,b_{t,\mu_t})$ and $\dist_\mathcal{H}(b,b_{t,\mu_t})$ is at least~$\ell$.
    Hence,
    \begin{align*}
        \abs{E(P)}+\dist_{H_{t,\mu_t-1}}(a,b_{t,\mu_t})
        &=1+\dist_{P_{t,\mu_t}}(b,b_{t,\mu_t})+\dist_{H_{t,\mu_t-1}}(a,b_{t,\mu_t})\\
        &\geq1+\dist_\mathcal{H}(b,b_{t,\mu_t})+\dist_\mathcal{H}(a,b_{t,\mu_t})\geq\ell+2,
    \end{align*}
    so that~$P$ is an $\ell$-extendable path of~$H_{t,\mu_t-1}$ shorter than~$P_{t,\mu_t}$, contradicting~\ref{cond:coarse2}.
    Thus, ${\{a,b\}\subseteq B_\mathcal{H}(b_{t,\mu_t},\ell-1)}$.

    \medskip
    \noindent\textbf{Case 2: $P_{t,\mu_t}$ is of length~$2$.}

    By Lemma~\ref{lem:lengthy} for~$H_{t,\mu_t-1}$ and~$P_{t,\mu_t}$, we have that~$t\geq2$, ${\mu_t=2}$, one of~$a_{t,2}$ and~$b_{t,2}$ is in~$V(H_{t-1,\mu_{t-1}})$, the other is in~$V(P_{t,1})$, and~$P_{t,1}$ is of type~2.
    By symmetry, we may assume that ${a_{t,2}\in V(H_{t-1,\mu_{t-1}})}$ and ${b_{t,2}\in V(P_{t,1})}$.

    We first assume that $a\in V(P_{t,1})$.
    By Lemma~\ref{lem:lengthy} for $H_{t,1}$, every $\ell$-extendable path of~$H_{t,1}$ having length~$2$ has one end in $H_{t-1,\mu_{t-1}}$.
    Therefore, $abb_{t,2}$ is not an $\ell$-extendable path of~$H_{t,1}$ as both~$a$ and~$b_{t,2}$ are in $V(P_{t,1})$.
    This implies that $\dist_{P_{t,1}}(a,b_{t,2})<\ell-2$, and therefore $\{a,b\}\subseteq B_\mathcal{H}(b_{t,2},\ell-1)$.

    We now assume that $a\notin V(P_{t,1})$.
    Observe that $a\in V(H_{t-1,\mu_{t-1}})$.
    By~\ref{cond:coarse3}, $G-Z_{t-1,\mu_{t-1}}$ has no $\ell$-extendable path of~$H_{t-1,\mu_{t-1}}$.
    Thus, $aba_{t,2}$ is not an $\ell$-extendable path of~$H_{t-1,\mu_{t-1}}$.
    Then the distance between~$a$ and~$a_{t,2}$ in~$H_{t-1,\mu_{t-1}}$ is less than~$\ell-2$, so ${\{a,b\}\subseteq B_\mathcal{H}(a_{t,2},\ell-1)}$.

    \medskip
    We conclude that there is $x\in\Gamma_{t,\mu_t}$ such that $\{a,b\}\subseteq B_\mathcal{H}(x,\ell-1)$.
\end{proof}

The next lemma considers an $\ell$-coarse ear-decomposition having no $\ell$-extendable path.

\begin{lemma}\label{lem:no edge2}
    Let~$G$ be a graph which has no $\ell$-cycle of length at most~$2\ell$ and let ${\mathcal{H}\coloneqq\bigcup_{i\in[t]}\bigcup_{j\in[\mu_i]}P_{i,j}}$ be an $\ell$-coarse ear-decomposition in~$G$ such that ${t\geq2}$ and ${G-Z_{t,\mu_t}}$ has no $\ell$-extendable path of~$\mathcal{H}$.
    If~$\mathcal{H}$ has a chord~$ab$ in~$G$ with $a\in V(H_{t-1,\mu_{t-1}})\setminus\Gamma_{t,1}$ and $b\in V(P_{t,1})\setminus\Gamma_{t,1}$, then there is $x\in\bigcup_{i=1}^{\min\{2,\mu_t\}}\Gamma_{t,i}$ such that $\{a,b\}\subseteq B_\mathcal{H}(x,\ell-1)$.
\end{lemma}
\begin{proof}
    Note that~$a$ is an admissible vertex of~$H_{t-1,\mu_{t-1}}$, because otherwise $b\in Z_{t-1,\mu_{t-1}}$.

    We divide into two cases depending on the type of~$P_{t,1}$.

    \medskip
    \noindent\textbf{Case 1: $P_{t,1}$ is of type~1.}

    Let~$P$ be the path obtained by concatenating~$ab$ and a shortest subpath of~$P_{t,1}$ between~$b$ and~$c_t$.
    Note that~$P$ is an $H_{t-1,\mu_{t-1}}$-path.

    We show that $\{a,b\}\subseteq B_\mathcal{H}(c_t,\ell-1)$.
    Suppose not.
    Then one of $\dist_\mathcal{H}(a,c_t)$ and $\dist_\mathcal{H}(b,c_t)$ is at least~$\ell$.
    Hence,
    \begin{align*}
        \abs{E(P)}+\dist_{H_{t-1,\mu_{t-1}}}(a,c_t)
        &=1+\dist_{P_{t,1}}(b,c_t)+\dist_{H_{t-1,\mu_{t-1}}}(a,c_t)\\
        &\geq1+\dist_\mathcal{H}(b,c_t)+\dist_\mathcal{H}(a,c_t)\geq\ell+2,
    \end{align*}
    so that~$P$ is an $\ell$-extendable path of~$H_{t-1,\mu_{t-1}}$, contradicting~\ref{cond:coarse3}.
    Thus, ${\{a,b\}\subseteq B_\mathcal{H}(c_t,\ell-1)}$.

    \medskip
    \noindent\textbf{Case 2: $P_{t,1}$ is of type~2.}

    Then $P_{t,1}$ has no branch vertex of~$H_{t,1}$, so that~$b$ is an admissible vertex of~$H_{t,1}$.
    Since~$a$ and~$b$ are in distinct components of~$H_{t,1}$, the edge~$ab$ forms an $\ell$-extendable path of~$H_{t,1}$.
    Since ${G-Z_{t,\mu_t}}$ has no $\ell$-extendable path of~$\mathcal{H}$, we have $\mu_t\geq2$.
    By~\ref{cond:coarse2}, $P_{t,2}$ is of length~$1$.
    By Lemma~\ref{lem:lengthy}, one of~$a_{t,2}$ and~$b_{t,2}$ is in~$V(H_{t-1,\mu_{t-1}})$ and the other is in~$V(P_{t,1})$.
    By symmetry, we may assume that ${a_{t,2}\in V(H_{t-1,\mu_{t-1}})}$ and ${b_{t,2}\in V(P_{t,1})}$.
    Since $ab\notin E(\mathcal{H})$, $\{a,b\}\neq\{a_{t,2},b_{t,2}\}$.

    We first assume that ${a=a_{t,2}}$.
    Let~$P'_{t,1}$ be the cycle obtained by concatenating~$bab_{t,2}$ and a shortest path of~$P_{t,1}$ between~$b$ and~$b_{t,2}$.
    Note that~$P'_{t,1}$ is a cycle in $G-Z_{t-1,\mu_{t-1}}$ intersecting~$V(H_{t-1,\mu_{t-1}})$ only at~$a$.
    Since~$P_{t,1}$ is of type~2, by~\ref{cond:coarse1}, $P'_{t,1}$ is not an $\ell$-cycle.
    This implies that $\dist_{P_{t,1}}(b,b_{t,2})<\ell-2$, and therefore, $\{a,b\}\subseteq B_\mathcal{H}(b_{t,2},\ell-1)$.

    We now assume that ${a\neq a_{t,2}}$.
    Let~$P$ be the path obtained by concatenating~$ab$, a shortest subpath of~$P_{t,1}$ between~$b$ and~$b_{t,2}$, and~$P_{t,2}$.
    By~\ref{cond:coarse3}, $G-Z_{t-1,\mu_{t-1}}$ has no $\ell$-extendable path of~$H_{t-1,\mu_{t-1}}$.
    Thus,~$P$ is not an $\ell$-extendable path of~$H_{t-1,\mu_{t-1}}$.
    This implies that
    \begin{align*}
        \ell-1
        &\geq\abs{E(P)}+\dist_{H_{t-1,\mu_{t-1}}}(a,a_{t,2})\\
        &=2+\dist_{P_{t,1}}(b,b_{t,2})+\dist_{H_{t-1,\mu_{t-1}}}(a,a_{t,2})\\
        &\geq2+\dist_\mathcal{H}(b,b_{t,2})+\dist_\mathcal{H}(a,a_{t,2}).
    \end{align*}
    Since~$a_{t,2}$ and~$b_{t,2}$ are adjacent, for any $x\in\Gamma_{t,2}$, $\{a,b\}\subseteq B_\mathcal{H}(x,\ell-1)$.
    
    \medskip
    We conclude that there is $x\in\Gamma_{t,1}\cup\Gamma_{t,2}$ such that $\{a,b\}\subseteq B_\mathcal{H}(x,\ell-1)$.
\end{proof}

We now prove Proposition~\ref{prop:induced}.

\begin{proof}[Proof of Proposition~\ref{prop:induced}]
    Let $(p,q)$ and~$(r,s)$ be the lexicographically smallest pairs of~$\mathcal{P}(\mathcal{H})$ such that ${a\in V(H_{p,q})}$ and ${b\in V(H_{r,s})}$.
    By symmetry, we may assume that $(p,q)\leq_L(r,s)$.

    We consider the following three cases.

    \medskip
    \noindent\textbf{Case 1: $a\in V(P_{r,s})$ and $s=1$.}

    Note that~$ab$ is a chord of~$P_{r,1}$ in~$G$.
    Since~$G$ has no $\ell$-cycle of length at most~$2\ell$, by Observation~\ref{obs:structure}\ref{item:type2}, $P_{r,1}$ is of type~1.
    Let~$P'_{r,1}$ be the cycle obtained by concatenating~$ab$ and a subpath of~$P_{r,1}$ between~$a$ and~$b$ which passes through~$c_p$.
    Note that~$P'_{r,1}$ is a cycle in~$G-Z_{r-1,\mu_{r-1}}$ intersecting~$V(H_{r-1,\mu_{r-1}})$ only at~$c_r$.
    In addition, it is shorter than~$P_{r,1}$, because~$ab$ is a chord of~$P_{r,1}$.
    Thus, by~\ref{cond:coarse1-1}, $P'_{r,1}$ is not an $\ell$-cycle.
    This implies that $\{a,b\}\subseteq B_\mathcal{H}(c_r,\ell-1)$.
    
    \medskip
    \noindent\textbf{Case 2: $a\in V(P_{r,s})$ and $s\geq2$.}

    Note that if $(p,q)<_L(r,s)$, then $a\in\Gamma_{r,s}$, and otherwise~$a$ is an internal vertex of~$P_{r,s}$.
    By symmetry, we may assume that~$a_{r,s}P_{r,s}a$ does not contain~$b$.
    Let~$P$ be the path obtained by concatenating~$a_{r,s}P_{r,s}a$, $ab$, and~$bP_{r,s}b_{r,s}$.
    Note that~$P$ is an $H_{r,s-1}$-path.
    In addition, it is shorter than~$P_{r,s}$, because~$ab$ is a chord of~$P_{r,s}$.
    By~\ref{cond:coarse2}, we have
    \[
        \abs{E(P)}+\dist_{H_{r,s-1}}(a_{r,s},b_{r,s})\leq\ell-1,
    \]
    that is, $\abs{E(P)}\leq\ell-2$.
    Hence, for any $x\in\Gamma_{r,s}$, $\{a,b\}\subseteq B_\mathcal{H}(x,\ell-1)$.

    \medskip
    \noindent\textbf{Case 3: $a\notin V(P_{r,s})$.}

    Note that $(p,q)<_L(r,s)$.
    Since~$(r,s)$ is the lexicographically smallest pair of~$\mathcal{P}(\mathcal{H})$ such that ${b\in V(H_{r,s})}$, we have that $b\in V(P_{r,s})\setminus\Gamma_{r,s}$.

    If $s\geq2$, then by Lemma~\ref{lem:no edge1} for $H_{r,s}$, there is $x\in\Gamma_{r,s}$ such that $\{a,b\}\subseteq B_\mathcal{H}(x,\ell-1)$.
    Thus, we may assume that~${s=1}$.
    Since $(p,q)<_L(r,s)$, we have that~${r\geq2}$.
    By the maximality of~$\mathcal{H}$, ${G-Z_{r,\mu_r}}$ has no $\ell$-extendable path of~$H_{r,\mu_r}$.
    Thus, by Lemma~\ref{lem:no edge2} for~$G$ and~$H_{r,\mu_r}$, there is $x\in\bigcup_{i=1}^{\min\{2,\mu_r\}}\Gamma_{r,i}$ such that $\{a,b\}\subseteq B_{H_{r,\mu_r}}(x,\ell-1)\subseteq B_\mathcal{H}(x,\ell-1)$.

    \medskip
    We conclude that~$\mathcal{H}$ has a branch vertex~$x$ such that $\{a,b\}\subseteq B_\mathcal{H}(x,\ell-1)$.
\end{proof}

Proposition~\ref{prop:induced} has the following corollary, which will be used to prove Theorem~\ref{thm:main1}.
We remark that if a graph~$G$ has no $\ell$-cycle of length at most~$2\ell$, then for an $\ell$-coarse ear-decomposition~$\mathcal{H}$ in~$G$ and branch vertices~$x,x'$ of~$\mathcal{H}$ with $\dist_\mathcal{H}(x,x')\leq\ell$, there is a unique shortest path of~$\mathcal{H}$ between~$x$ and~$x'$, because otherwise~$\mathcal{H}$ has a cycle of length at most~$2\ell$ which is an $\ell$-cycle in~$G$ by Lemma~\ref{lem:ell cycle}.

\begin{corollary}\label{cor:induced}
    Let~$G$ be a graph which has no $\ell$-cycle of length at most~$2\ell$, let ${\mathcal{H}\coloneqq\bigcup_{i\in[t]}\bigcup_{j\in[\mu_i]}P_{i,j}}$ be a maximal $\ell$-coarse ear-decomposition in~$G$, and let~$\mathcal{C}$ be an induced packing of cycles in~$\mathcal{H}$.
    If~$\mathcal{H}$ has a chord in~$G$ between distinct cycles~$C$ and~$C'$ in~$\mathcal{C}$, then~$\mathcal{H}$ has distinct branch vertices~$x$ and~$x'$ such that $\dist_\mathcal{H}(x,x')\leq\ell-1$ and for the shortest path~$Q$ of~$\mathcal{H}$ between~$x$ and~$x'$, $B_{\mathcal{H}-E(Q)}(x,\ell-1)\subseteq V(C)$ and $B_{\mathcal{H}-E(Q)}(x',\ell-1)\subseteq V(C')$.
\end{corollary}
\begin{proof}
    By Proposition~\ref{prop:induced}, $\mathcal{H}$ has a branch vertex~$x$ such that both ends of the chord are in~${B_\mathcal{H}(x,\ell-1)}$.
    If~$\mathcal{H}$ has no other branch vertex contained in $B_\mathcal{H}(x,\ell-1)$, then both~$C$ and~$C'$ must contain~$x$, contradicting the fact that~$C$ and~$C'$ are vertex-disjoint.
    Hence,~$\mathcal{H}$ has another branch vertex~$x'$ contained in $B_\mathcal{H}(x,\ell-1)$.

    Since $\dist_\mathcal{H}(x,x')\leq\ell-1$, by Corollary~\ref{cor:close ends}, there is a unique $(p,q)\in\mathcal{P}(\mathcal{H})$ such that $\{x,x'\}=\{a_{p,q},b_{p,q}\}$.
    Then~$x$ and~$x'$ are the unique branch vertices of~$\mathcal{H}$ in ${B_\mathcal{H}(\Gamma_{p,q},\ell-1)}$.
    This implies that~$x$ is a vertex of one of~$C$ and~$C'$, and~$x'$ is a vertex of the other.
    By symmetry, we may assume that $x\in V(C)$ and $x'\in V(C')$.

    Since~$Q$ is of length at most~$\ell-1$, every internal vertex of~$Q$ has degree~$2$ in~$\mathcal{H}$.
    Since $x\in V(C)$ and $x'\in V(C')$, no internal vertex of~$Q$ is contained in $V(C\cup C')$.
    Therefore, we have $B_{\mathcal{H}-E(Q)}(x,\ell-1)\subseteq V(C)$ and $B_{\mathcal{H}-E(Q)}(x',\ell-1)\subseteq V(C')$.
\end{proof}

\section{A proof of Theorem~\ref{thm:main1}}
\label{sec:maintheorem}

In this section, we prove Theorem~\ref{thm:main1}.

\mainone*

To prove this, we construct two auxiliary graphs of an $\ell$-coarse ear-decomposition, and derive their properties.
The first one is defined as follows.

For an $\ell$-coarse ear-decomposition $\mathcal{H}\coloneqq\bigcup_{i\in[t]}\bigcup_{j\in[\mu_i]}P_{i,j}$ in a graph~$G$, let~$O_\mathcal{H}$ be the graph with vertex set ${\{w_i \colon i\in[t]\}}$ such that distinct~$w_i$ and~$w_j$ are adjacent in~$O_\mathcal{H}$ if and only if the distance between~$V(P_{i,1})$ and~$V(P_{j,1})$ in~$\mathcal{H}$ is at most~$\ell-1$.

We show that $O_\mathcal{H}$ is a forest when~$G$ has no $\ell$-cycle of length at most~$2\ell$, and from this, derive that if~$t$ is large, then~$G$ has a large induced packing of $\ell$-cycles.

\begin{lemma}\label{lem:packing1}
    Let~$G$ be a graph which has no $\ell$-cycle of length at most~$2\ell$ and let ${\mathcal{H}\coloneqq\bigcup_{i\in[t]}\bigcup_{j\in[\mu_i]}P_{i,j}}$ be a maximal $\ell$-coarse ear-decomposition in~$G$.
    Then~$O_\mathcal{H}$ is a forest.
    In addition, if~${t\geq2k-1}$, then~$G$ has an induced packing of~$k$ $\ell$-cycles.
\end{lemma}
\begin{proof}
    To show that $O_\mathcal{H}$ is a forest, it suffices to show that for each ${i\in[t]}$, $w_i$ has at most one neighbour in ${\{w_j \colon j\in[i-1]\}}$.
    Towards a contradiction, suppose that there exists ${p\in[t]}$ such that~$w_p$ has distinct neighbours $w_{p_1}$ and $w_{p_2}$ in~$O_\mathcal{H}$ for some $p_1,p_2\in[p-1]$.

    \begin{claim}\label{clm:type2}
        $P_{p,1}$ is of type~2.
    \end{claim}
    \begin{subproof}
        Since $V(P_{p,1})$ intersects at most one of~$V(P_{p_1,1})$ and~$V(P_{p_2,1})$, without loss of generality, we may assume that~$P_{p_1,1}$ and~$P_{p,1}$ are vertex-disjoint.
        Since~$w_p$ is adjacent to~$w_{p_1}$ in~$O_\mathcal{H}$, $\mathcal{H}$ has a $(P_{p_1,1},P_{p,1})$-path~$Q$ of length at most~$\ell-1$ between~$a\in V(P_{p_1,1})$ and~$b\in V(P_{p,1})$.
        Note that~$a$ is an admissible vertex of~$H_{p-1,\mu_{p-1}}$, because otherwise $b\in Z_{p-1,\mu_{p-1}}$.
        Since~$Q$ and $P_{p_1,1}\cup P_{p,1}$ are edge-disjoint, both~$a$ and~$b$ are branch vertices of~$\mathcal{H}$.
        Then by Corollary~\ref{cor:close ends}, there is a unique $(i,j)\in\mathcal{P}(\mathcal{H})$ such that $\{a,b\}=\{a_{i,j},b_{i,j}\}$.
        By Lemma~\ref{lem:unique}, both~$a$ and~$b$ have degree~$3$ in~$\mathcal{H}$.
        This implies that $(i,j)>_L(p,1)$.

        Since~$Q$ is of length at most~$\ell-1$, if $Q=P_{i,j}$, then~$P_{p,1}$ is of type~2 by Lemma~\ref{lem:lengthy} for~$H_{i,j-1}$ and~$P_{i,j}$.
        Thus, we may assume that $Q\neq P_{i,j}$.
        Then~$P_{i,j}$ has an internal vertex~$c$ such that $aP_{i,j}c=aQc$ and~$c$ is incident with edges $e\in E(P_{i,j})\setminus E(Q)$ and $e'\in E(Q)\setminus E(P_{i,j})$.
        Let $(i',j')$ be the unique pair in~$\mathcal{P}(\mathcal{H})$ such that $e'\in E(P_{i',j'})$.
        Since~$c$ is an internal vertex of~$P_{i,j}$, we deduce that $(i',j')>_L(i,j)$.
        Since~$Q$ is of length at most~$\ell-1$, the distance between~$a$ and~$c$ in~$P_{i,j}$ is at most~$\ell-2$.
        Therefore, one of the ends of~$e'$ is contained in~$Z_{i,j}$, a contradiction, and this proves the claim.
    \end{subproof}

    By Claim~\ref{clm:type2}, $V(P_{p,1})$ is disjoint from $V(P_{p_1,1})\cup V(P_{p_2,1})$.
    Thus, for each $i\in[2]$, $\mathcal{H}$ has a $(P_{p_i,1},P_{p,1})$-path~$Q_i$ of length at most~$\ell-1$ between~$a_i\in V(P_{p_i,1})$ and~$b_i\in V(P_{p,1})$.
    Recall that both~$a_1$ and~$a_2$ are admissible vertices of~$H_{p-1,\mu_{p-1}}$, and that there are unique $(i_1,j_1),(i_2,j_2)\in\mathcal{P}(\mathcal{H})$ such that $\{a_1,b_1\}=\{a_{i_1,j_1},b_{i_1,j_1}\}$ and $\{a_2,b_2\}=\{a_{i_2,j_2},b_{i_2,j_2}\}$.
    Since~$P_{p_1,1}$ and~$P_{p_2,1}$ are vertex-disjoint, $(i_1,j_1)$ and~$(i_2,j_2)$ are distinct.
    By symmetry, we may assume that $(i_1,j_1)<_L(i_2,j_2)$.
    Similar to the proof of Claim~\ref{clm:type2}, we deduce that ${Q_1=P_{i_1,j_1}}$ and ${Q_2=P_{i_2,j_2}}$.
    Since both~$Q_1$ and~$Q_2$ are of length at most~$\ell-1$, they are vertex-disjoint, because otherwise $Z_{i_1,j_1}\cap V(P_{i_2,j_2})\neq\emptyset$.
    
    Let~$P$ be a path obtained by concatenating~$Q_1$, $Q_2$, and any subpath of~$P_{p,1}$ between~$b_1$ and~$b_2$.
    Note that~$P$ is an $H_{p-1,\mu_{p-1}}$-path in ${G-Z_{p-1,\mu_{p-1}}}$.
    Since $(i_1,j_1)\neq (i_2,j_2)$, by Corollary~\ref{cor:close ends}, $\dist_\mathcal{H}(b_1,b_2)>\ell-1$.
    Thus, $\abs{E(P)}>\ell+1$, and therefore~$P$ is an $\ell$-extendable path of~$H_{p-1,\mu_{p-1}}$, contradicting~\ref{cond:coarse3}.
    Therefore, $O_\mathcal{H}$ is a forest.

    We now suppose that~${t\geq2k-1}$.
    Since $O_\mathcal{H}$ is a forest, $O_\mathcal{H}$ has an independent set~$I$ of size~$k$.
    By the definition of~$O_\mathcal{H}$, $\{P_{i,1} \colon w_i\in I\}$ is an induced packing of~$k$ $\ell$-cycles in~$\mathcal{H}$.
    
    We show that $\{P_{i,1} \colon w_i\in I\}$ is an induced packing of cycles in~$G$.
    Suppose not.
    Then~$\mathcal{H}$ has a chord in~$G$ between $P_{i,1}$ and~$P_{i',1}$ for some distinct $w_i,w_{i'}\in I$.
    By Corollary~\ref{cor:induced}, $\mathcal{H}$ has distinct branch vertices~$x$ and~$x'$ such that ${\dist_\mathcal{H}(x,x')\leq\ell-1}$ and for the shortest path~$Q$ of~$\mathcal{H}$ between~$x$ and~$x'$, $B_{\mathcal{H}-E(Q)}(x,\ell-1)\subseteq V(P_{i,1})$ and $B_{\mathcal{H}-E(Q)}(x',\ell-1)\subseteq V(P_{i',1})$.
    Then $\dist_\mathcal{H}(V(P_{i,1}),V(P_{i',1}))\leq\dist_\mathcal{H}(x,x')\leq\ell-1$, contradicting that~$w_i$ and~$w_{i'}$ are not adjacent in~$O_\mathcal{H}$.
    Hence, $\{P_{i,1} \colon w_i\in I\}$ is an induced packing of~$k$ $\ell$-cycles in~$G$.
\end{proof}

We now construct the second auxiliary graph.
Let~$\mathcal{P}'(\mathcal{H})$ be the set of pairs $(i,j)\in\mathcal{P}(\mathcal{H})$ such that $j\geq2$ and $\dist_{H_{i,j-1}}(a_{i,j},b_{i,j})\leq\ell-1$.
Let~$U_\mathcal{H}$ be the graph with vertex set ${\{u_{i,j}:(i,j)\in\mathcal{P}'(\mathcal{H})\}}$ such that~$u_{i,j}$ and~$u_{i',j'}$ with $(i,j)>_L(i',j')$ are adjacent in~$U_\mathcal{H}$ if and only if $\Gamma_{i,j}\subseteq V(P_{i',j'})$ and $\dist_{P_{i',j'}}(a_{i,j},b_{i,j})\leq\ell-1$.
Note that each vertex $u_{i,j}$ of~$U_\mathcal{H}$ is adjacent to at most one vertex $u_{i',j'}$ with~$(i',j')<_L(i,j)$, and therefore $U_\mathcal{H}$ is a forest.

We show that if~$U_\mathcal{H}$ has many vertices, then~$\mathcal{H}$ has a large induced packing of $\ell$-cycles.

\begin{lemma}\label{lem:packing2}
    Let~$G$ be a graph which has no $\ell$-cycle of length at most~$2\ell$ and let ${\mathcal{H}\coloneqq\bigcup_{i\in[t]}\bigcup_{j\in[\mu_i]}P_{i,j}}$ be a maximal $\ell$-coarse ear-decomposition in~$G$.
    If~$U_\mathcal{H}$ has at least $2k-1$ vertices, then~$G$ has an induced packing of~$k$ $\ell$-cycles.
\end{lemma}
\begin{proof}
    For each $(i,j)\in\mathcal{P}'(\mathcal{H})$, let~$C_{i,j}$ be the cycle obtained by concatenating~$P_{i,j}$ and a shortest path of~$H_{i,j-1}$ between~$a_{i,j}$ and~$b_{i,j}$.
    Since $U_\mathcal{H}$ is a forest, if $U_\mathcal{H}$ has at least~$2k-1$ vertices, then $U_\mathcal{H}$ has an independent set~$I$ of size~$k$.
    Let $\mathcal{C}\coloneqq\{C_{i,j} \colon u_{i,j}\in I\}$.

    We first show that the cycles in~$\mathcal{C}$ are vertex-disjoint.
    Suppose not.
    Then~$\mathcal{C}$ contains cycles~$C_{i,j}$ and~$C_{i',j'}$ such that $(i,j)>_L(i',j')$ and $V(C_{i,j})\cap V(C_{i',j'})\neq\emptyset$.
    Let
    \begin{align*}
        Q_{i,j}&\coloneqq C_{i,j}-(V(P_{i,j})\setminus\Gamma_{i,j}),\\
        Q_{i',j'}&\coloneqq C_{i',j'}-(V(P_{i',j'})\setminus\Gamma_{i',j'}),
    \end{align*}
    and let~$z$ be a vertex in $V(C_{i,j})\cap V(C_{i',j'})$.
    Since $(i,j)>_L(i',j')$, no vertex in $V(P_{i,j})\setminus\Gamma_{i,j}$ is in~$V(C_{i',j'})$, and therefore~$z$ is a vertex of~$Q_{i,j}$.
    Since the length of~$Q_{i,j}$ is at most~$\ell-1$, $z$ is an admissible vertex of~$H_{i,j-1}$, because otherwise $\Gamma_{i,j}\subseteq Z_{i,j-1}$.
    This implies that $z\notin V(Q_{i',j'})$, as $V(Q_{i',j'})\subseteq B_{H_{i,j-1}}(\Gamma_{i',j'},\ell-1)$.
    Therefore, $z\in V(P_{i',j'})\setminus\Gamma_{i',j'}$.
    Since~$z$ is an admissible vertex of~$H_{i,j-1}$, $B_{H_{i,j-1}}(z,\ell-1)$ induces a path in~$H_{i,j-1}$.
    Therefore, both~$a_i$ and~$b_i$ are vertices of~$P_{i',j'}$, contradicting~$u_{i,j}$ and~$u_{i',j'}$ are not adjacent in~$U_\mathcal{H}$.
    Hence, the cycles in~$\mathcal{C}$ are vertex-disjoint.
    
    We now show that~$\mathcal{C}$ is an induced packing of cycles in~$G$.
    Suppose not.
    Since the cycles in~$\mathcal{C}$ are vertex-disjoint, $\mathcal{H}$ has a chord in~$G$ between distinct cycles~$C_{i,j}$ and~$C_{i',j'}$ in~$\mathcal{C}$.
    By Corollary~\ref{cor:induced}, $\mathcal{H}$ has distinct branch vertices~$x$ and~$x'$ such that ${\dist_\mathcal{H}(x,x')\leq\ell-1}$ and for the shortest path~$Q$ of~$\mathcal{H}$ between~$x$ and~$x'$, ${B_{\mathcal{H}-E(Q)}(x,\ell-1)\subseteq V(C_{i,j})}$ and ${B_{\mathcal{H}-E(Q)}(x',\ell-1)\subseteq V(C_{i',j'})}$.
    By Corollary~\ref{cor:close ends}, there is $(p,q)\in\mathcal{P}(\mathcal{H})$ such that ${\{x,x'\}=\{a_{p,q},b_{p,q}\}}$.
    In addition, $x$ and~$x'$ are the only branch vertices in $B_\mathcal{H}(\Gamma_{p,q},\ell-1)$.
    Hence, $x\in V(P_{i,j})\setminus B_\mathcal{H}(\Gamma_{i,j},\ell-1)$ and $x'\in V(P_{i',j'})\setminus B_\mathcal{H}(\Gamma_{i',j'},\ell-1)$.
    Thus, any path of~$\mathcal{H}-E(P_{p,q})$ between~$x$ and~$x'$ must contain a branch vertex of~$\mathcal{H}$, and therefore $P_{p,q}=Q$.
    This contradicts Lemma~\ref{lem:lengthy} as $j,j'\geq2$ and~$P_{p,q}$ is of length at most~$\ell-1$.
\end{proof}

We now prove Theorem~\ref{thm:main1}.

\begin{proof}[Proof of Theorem~\ref{thm:main1}]
    We proceed by induction on~$k$.
    The statement obviously holds for~${k=1}$ by taking both~$f(1,\ell)$ and~$g(1)$ as~$0$.
    Thus, we may assume that~${k\geq2}$.
    Let
    \begin{align*}
        f(k,\ell)&\coloneqq (4\ell-3)s_k+(80\ell-60)(k-1)=\mathcal{O}(\ell k\log k),\\
        g(k)&\coloneqq s_k+20(k-1)=\mathcal{O}(k\log k).
    \end{align*}
    We remark that for every integer $k\geq2$, $f(k-1,\ell)+2\ell\leq f(k,\ell)$ and $g(k-1)+2\leq g(k)$.
    
    If~$G$ has no $\ell$-cycle, then the statement clearly holds by taking both~$X_1$ and~$X_2$ as empty sets.
    Thus, we may assume that~$G$ has at least one $\ell$-cycle.
    
    Suppose that~$G$ has an $\ell$-cycle~$C$ of length at most~$2\ell$.
    By the inductive hypothesis on ${G'\coloneqq G-B_G(C,1)}$ with~${k-1}$, ${G-B_G(C,1)}$ contains either an induced packing~$\mathcal{C}$ of~$k-1$ $\ell$-cycles or sets~$X'_1$ and~$X'_2$ of vertices with ${\abs{X'_1} \leq f(k-1,\ell)}$ and ${\abs{X'_2} \leq g(k-1)}$ such that neither~${G'-B_{G'}(X'_1,1)}$ nor~${G'-B_{G'}(X'_2,\ell)}$ has an $\ell$-cycle.
    In the former, $\mathcal{C}\cup\{C\}$ is an induced packing of~$k$ $\ell$-cycles in~$G$.
    Thus, we may assume the latter case.
    Since~$C$ is of length at most~$2\ell$, there are two vertices~$v$ and~$v'$ of~$C$ such that $V(C)\subseteq B_G(\{v,v'\},\ell-1)$ so that $B_G(C,1)\subseteq B_G(\{v,v'\},\ell)$.
    Then $X'_1\cup V(C)$ and $X'_2\cup\{v,v'\}$ are desired sets as 
    \begin{align*}
        \abs{X'_1\cup V(C)}&\leq f(k-1,\ell)+2\ell\leq f(k,\ell)\\
        \abs{X'_2\cup\{v\}}&\leq g(k-1)+2\leq g(k),
    \end{align*}
    so the statement holds.
    Hence, we may assume that~$G$ has no $\ell$-cycle of length at most~$2\ell$.
    
    Let $\mathcal{H}\coloneqq\bigcup_{i\in[t]}\bigcup_{j\in[\mu_i]}P_{i,j}$ be a maximal $\ell$-coarse ear-decomposition in~$G$, let~$\beta$ be the number of branch vertices of~$\mathcal{H}$, and let~$\mathcal{P}'(\mathcal{H})$ be the set of pairs $(i,j)\in\mathcal{P}(\mathcal{H})$ such that $j\geq2$ and ${\dist_{H_{i,j-1}}(a_{i,j},b_{i,j})\leq\ell-1}$.
    Since~$G$ has no $\ell$-cycle of length at most~$2\ell$, if ${t\geq2k-1}$ or ${\abs{\mathcal{P}'(\mathcal{H})}\geq2k-1}$, then by Lemmas~\ref{lem:packing1} or~\ref{lem:packing2}, $G$ has an induced packing of~$k$ $\ell$-cycles.
    Thus, we may assume that ${t\leq2k-2}$ and~${\abs{\mathcal{P}'(\mathcal{H})} \leq 2k-2}$.

    Let~$\mathcal{H}'$ be the graph obtained from~$\mathcal{H}$ as follows: remove every vertex~$v$ such that either 
    \begin{itemize}
        \item $\deg_\mathcal{H}(v)=4$, or
        \item $v=a_{i,j}$ for some ${(i,j)\in\{(x,2) \colon x\in[t],\abs{E(P_{x,2})}\leq\ell-1\}}$, or 
        \item $v=a_{i,j}$ for some $(i,j)\in\mathcal{P}'(\mathcal{H})$, 
    \end{itemize} 
    and then recursively remove vertices of degree at most~$1$.
    Note that every vertex of~$\mathcal{H}'$ has degree~$2$ or~$3$ in~$\mathcal{H}'$.
    We remark that throughout the process, if we remove a degree-$4$ vertex~$v$ from~$\mathcal{H}$, then~$\mathcal{H}'$ loses at most five branch vertices of~$\mathcal{H}$ including~$v$, and if we remove~$a_{i,j}$ for ${(i,j)\in\mathcal{P}'(\mathcal{H})\cup\{(x,2) \colon x\in[t], \abs{E(P_{x,2})} \leq\ell-1\}}$ from~$\mathcal{H}$, then~$\mathcal{H}'$ loses at most four branch vertices of~$\mathcal{H}$ including~$a_{i,j}$.
    Also, note that if $\deg_{\mathcal{H}}(v)=4$, then $v=c_x$ for some ${x\in[t]}$ and~$P_{x,1}$ is of type~1.
    In this case, by Lemma~\ref{lem:lengthy}, we have that ${\abs{E(P_{x,2})} > \ell-1}$.
    Thus, for the number~$m$ of degree-$4$ vertices of~$\mathcal{H}$, there are at most $t-m$ integers in~$[t]$ such that $\abs{E(P_{x,2})}\leq\ell-1$.
    Since ${t\leq2k-2}$ and~${\abs{\mathcal{P}'(\mathcal{H})} \leq 2k-2}$, the number of branch vertices of~$\mathcal{H}'$ is at least 
    \[
        \beta-5m-4(2k-2-m)-4(2k-2)=\beta-16(k-1)-m\geq\beta-18(k-1).
    \]

    We divide into two cases depending on the number of branch vertices of~$\mathcal{H}'$.
    
    \medskip
    \noindent\textbf{Case 1: $\mathcal{H}'$ has at least~$s_k$ branch vertices.}
    
    By Theorem~\ref{thm:simonovitz}, there are~$k$ vertex-disjoint cycles $C_1,\ldots,C_k$ in~$\mathcal{H}'$.
    Since~$\mathcal{H}'$ is a subgraph of~$\mathcal{H}$, every cycle in~$\mathcal{L}\coloneqq\{C_1,\ldots,C_k\}$ is an $\ell$-cycle in~$G$ by Lemma~\ref{lem:ell cycle}.
    We are going to show that $\mathcal{L}$ is an induced packing of cycles in~$G$.

    We first show that $\mathcal{L}$ is an induced packing of cycles in~$\mathcal{H}'$.
    Suppose not.
    Then~$\mathcal{H}$ has an edge~$ab$ between distinct cycles~$C_i$ and~$C_j$ in~$\mathcal{L}$.
    By the construction of~$\mathcal{H}'$, both~$a$ and~$b$ have degree~$3$ in~$\mathcal{H}$.
    By Corollary~\ref{cor:close ends}, there exists a unique $(p,q)\in\mathcal{P}(\mathcal{H})$ such that ${\{a,b\}=\{a_{p,q},b_{p,q}\}}$.
    Since~$P_{p,q}$ is of length~$1$, by Lemma~\ref{lem:lengthy} for~$H_{p,q-1}$ and~$P_{p,q}$, we have ${q=2}$.
    Then~$a_{p,q}$ was deleted when we constructed $\mathcal{H}'$ from $\mathcal{H}$, a contradiction.
    Hence, $\mathcal{L}$ is an induced packing of cycles in~$\mathcal{H}'$.
    
    We now show that~$\mathcal{L}$ is an induced packing of cycles in~$G$.
    Suppose not.
    Since~$\mathcal{L}$ is an induced packing of cycles in~$\mathcal{H}'$, there is a chord of~$\mathcal{H}'$ in~$G$ between distinct cycles~$C_i$ and~$C_j$ in~$\mathcal{L}$.
    Since~$\mathcal{H}'$ is an induced subgraph of~$\mathcal{H}$, the chord is also a chord of~$\mathcal{H}$ in~$G$.
    By Corollary~\ref{cor:induced}, $\mathcal{H}$ has distinct branch vertices~$x$ and~$x'$ such that ${\dist_\mathcal{H}(x,x')\leq\ell-1}$ and for the shortest path~$Q$ of~$\mathcal{H}$ between~$x$ and~$x'$, ${B_{\mathcal{H}-E(Q)}(x,\ell-1)\subseteq V(C_i)}$ and ${B_{\mathcal{H}-E(Q)}(x',\ell-1)\subseteq V(C_j)}$.
    By Corollary~\ref{cor:close ends}, there is $(p,q)\in\mathcal{P}(\mathcal{H})$ such that ${\{x,x'\}=\{a_{p,q},b_{p,q}\}}$.

    Since $(p,q)\notin\mathcal{P}'(\mathcal{H})$, the distance between~$x$ and~$x'$ in~$H_{p,q-1}$ is longer than~$\ell-1$.
    By Corollary~\ref{cor:close ends}, $x$ and~$x'$ are the only branch vertices in $B_\mathcal{H}(\Gamma_{p,q},\ell-1)$, and therefore~$P_{p,q}$ is the only path of~$\mathcal{H}$ between~$x$ and~$x'$ which has no branch vertex of~$\mathcal{H}$ as an internal vertex.
    This implies that~$P_{p,q}$ is the unique path of~$\mathcal{H}$ between~$x$ and~$x'$ whose length is at most~$\ell-1$.
    By Lemma~\ref{lem:lengthy} for~$H_{p,q-1}$ and~$P_{p,q}$, we have $q=2$.
    Then~$a_{p,q}$ was deleted when we constructed $\mathcal{H}'$ from $\mathcal{H}$, a contradiction.
    
    Hence, $\mathcal{L}$ is an induced packing of $k$ $\ell$-cycles in~$\mathcal{H}'$.
        
    \medskip
    \noindent\textbf{Case 2: $\mathcal{H}'$ has less than~$s_k$ branch vertices.}
    
    Then~$\mathcal{H}$ has less than $s_k+18(k-1)$ branch vertices.
    Let~$X_2$ be the set obtained from the set of branch vertices of~$\mathcal{H}$ by adding one arbitrary vertex from each component of~$\mathcal{H}$ which is a cycle, and let $X_1\coloneqq B_\mathcal{H}(X_2,\ell-1)$.
    Since~$\mathcal{H}$ has at most~$t$ components, we have
    \[
        \abs{X_2}<s_k+18(k-1)+t\leq s_k+20(k-1)=g(k).
    \]    
    For each vertex~$v$ of~$\mathcal{H}$, $B_\mathcal{H}(v,\ell-1)$ contains at most $4(\ell-1)+1$ vertices.
    Thus,
    \[
        \abs{X_1}\leq(4\ell-3)\abs{X_2}<(4\ell-3)s_k+(80\ell-60)(k-1)=f(k,\ell).
    \]
    Note that every component of $\mathcal{H}-X_2$ has maximum degree at most~$2$ and is not a cycle.
    Thus, the components of $\mathcal{H}-X_2$, say $H_1,\ldots,H_x$, are paths.

    We show that neither ${G-B_G(X_1,1)}$ nor ${G-B_G(X_2,\ell)}$ has an $\ell$-cycle.
    Since~$B_G(X_1,1)$ is a subset of~$B_G(X_2,\ell)$, it suffices to show that ${G'\coloneqq G-B_G(X_1,1)}$ has no $\ell$-cycle.
    Towards a contradiction, suppose that~$G'$ has an $\ell$-cycle.
    Since $Z_{t,\mu_t}\subseteq B_G(X_1,1)$, by the maximality of~$\mathcal{H}$, every $\ell$-cycle in~$G'$ intersects at least one of~$H_1,\ldots,H_x$.
    
    We first show that every $\ell$-cycle in~$G'$ intersects exactly one of~$H_1,\ldots,H_x$.
    It suffices to show that every component of~$G'$ contains at most one of~$H_1,\ldots,H_x$ as subgraphs.
    Suppose not.
    Then there are distinct $i,j\in[x]$ such that~$G'$ has an $(H_i,H_j)$-path~$Q$ between $q_i\in V(H_i)$ and $q_j\in V(H_j)$.
    We may assume that no internal vertex of~$Q$ is a vertex of $\mathcal{H}-X_1$, because otherwise we can take a shorter path between two of $H_1,\ldots,H_x$.
    Note that $q_j\notin V(H_i)$ as~$Q$ is an $(H_i,H_j)$-path.
    Then either~$q_i$ and~$q_j$ are in distinct components of~$\mathcal{H}$, or a shortest path of~$\mathcal{H}$ between~$q_i$ and~$q_j$ contains a branch vertex of~$\mathcal{H}$ so that ${\dist_\mathcal{H}(q_i,q_j)\geq2\ell}$.
    Therefore, $Q$ is an $\ell$-extendable path of~$\mathcal{H}$, contradicting the maximality of~$\mathcal{H}$.
    Hence, every $\ell$-cycle in~$G'$ intersects exactly one of~$H_1,\ldots,H_x$.

    Among all shortest $\ell$-cycles in~$G'$, we choose~$C$ minimising the number of its subpaths which are $\mathcal{H}$-paths.
    Note that the length of~$C$ is at least ${2\ell+1}$.
    Without loss of generality, we may assume that~$C$ intersects~$H_1\coloneqq v_1\cdots v_h$.
    Since~$H_1$ is a path, $C$ has a subpath~$Q$ which is an $\mathcal{H}$-path between distinct~$v_i$ and~$v_j$.
    Note that every path of~$\mathcal{H}$ between~$v_i$ and~$v_j$ other than~$v_iH_1v_j$ must contain a vertex in~$X_2$, so ${\dist_\mathcal{H}(v_i,v_j)\geq\min\{\abs{E(v_iH_1v_j)},2\ell\}}$.
    Then by the maximality of~$\mathcal{H}$, we deduce that ${\abs{E(Q)}+\abs{E(v_iH_1v_j)}<\ell}$.
    Since~$C$ is of length at least~$2\ell+1$, ${v_iH_1v_j}$ has a subpath~$R$ with ends~$z$ and~$z'$ which is a $C$-path.

    Since~$C$ is of length at least $2\ell+1$, it has a subpath~$D$ between~$z$ and~$z'$ of length at least~$\ell+1$.
    Let~$D'$ be the other subpath of~$C$ between~$z$ and~$z'$.
    Since~$C$ is a shortest $\ell$-cycle in~$G'$, we have ${\abs{E(D')}\leq\abs{E(R)}\leq\ell-2}$.
    If~$D'$ is shorter than~$R$, then the walk obtained from~$H_1$ by replacing~$R$ with~$D'$ contains a path of~$G'$ between~$v_1$ and~$v_h$ which is shorter than~$H_1$, contradicting~\ref{cond:coarse1} or~\ref{cond:coarse2}.
    Hence, $D'$ and~$R$ are of the same length.

    Since~$R$ is a subpath of~$H_1$, we deduce that~$D'$ has at least one subpath which is an $\mathcal{H}$-path.
    Thus, by replacing~$D'$ with~$R$, we obtain a shortest $\ell$-cycle in~$G'$ such that the number of its subpaths which are $\mathcal{H}$-paths is at least one less than that of~$C$, contradicting the choice of~$C$.
    Hence, $G'$ has no $\ell$-cycle, and therefore both~$X_1$ and~$X_2$ are desired sets.

    \medskip
    This completes the proof by induction.
\end{proof}

We now argue how to modify the proof of Theorem~\ref{thm:main1} to derive Theorem~\ref{thm:main2}.
We may assume that~$G\in\mathcal{C}$ has no $\ell$-cycle of length at most~$2\ell$.
As in the proof of Theorem~\ref{thm:main1}, we construct the same graph~$\mathcal{H}'$ from~$G$.
Since~$\mathcal{C}$ is closed under taking subgraphs, $\mathcal{H}'$ is also contained in~$\mathcal{C}$.
Recall that every cycle in~$\mathcal{H}'$ is an $\ell$-cycle in~$G$.
If~$\mathcal{H}'$ contains~$k$ vertex-disjoint cycles, then we can show that those cycles form an induced packing of $\ell$-cycles in~$G$ as we did in the proof of Theorem~\ref{thm:main1} when $\mathcal{H}'$ has at least~$s_k$ branch vertices.
If~$\mathcal{H}'$ has no~$k$ vertex-disjoint cycles, then it contains a set~$X'$ of at most~$h(k)$ vertices such that $\mathcal{H}'-X'$ is a forest.
By Lemma~\ref{lem:cubic}, $\mathcal{H}'$ contains less than $7h(k)-1$ branch vertices.
Therefore, $\mathcal{H}$ has less than ${7h(k)+18(k-1)-1}$ branch vertices.
We choose~$X_1$ and~$X_2$ the same as in the proof of Theorem~\ref{thm:main1} and then show that neither ${G-B_G(X_1,1)}$ nor ${G-B_G(X_2,\ell)}$ has an $\ell$-cycle.
Since
\begin{align*}
    \abs{X_1} 
    &<(4\ell-3)(7h(k)+18(k-1)-1) + t\\
    &\leq (28\ell-21)h(k) + (74\ell-54)k-78\ell+57,\\
    \abs{X_2}
    &<7h(k)+18(k-1)-1+t\leq7h(k)+20k-21,
\end{align*}
this completes the proof of Theorem~\ref{thm:main2}.

\section{A tighter bound for planar graphs}
\label{sec:planar}

In this section, we prove Theorem~\ref{thm:main3}.

\planar*

\begin{proof}
    We proceed by induction on~$k$. The statement clearly holds for~${k=1}$.
    Thus, we may assume that~${k\geq2}$. Suppose that $G$ has no induced packing of~$k$ cycles.
    We may assume that~$G$ contains a cycle, otherwise we could take $X\coloneqq\emptyset$.
    
    Let~$G'$ be the multigraph obtained from~$G$ by suppressing all vertices of degree at most~$2$.
    By Euler's formula, $G'$ contains a cycle~$C'$ of length at most~$5$.
    Let~$C$ be the corresponding cycle in~$G$ and let~$A$ be the set of vertices $v \in V(C)$ such that $\deg_G(v) \geq 3$.
    Since~$C'$ has length at most~$5$, we have $\abs{A}\leq 5$.  
    
    Let $G^*:=G-B_G(A,1)$ and let~$D$ be a cycle in~$G^*$.
    Suppose that $x \in V(D)$ and $y \in V(C)$ with $xy \in E(G)$.
    Since $\deg_G(v)=2$ for all $v \in V(C) \setminus A$, we must have $y \in A$.
    Thus, $x \in B_G(A,1)$, which is a contradiction.
    Therefore,~$G^*$ has no induced packing of~$k-1$ cycles, and by induction, $G^*$ has a set~$Y$ of size at most $5(k-1)$ such that $G^*-B_{G^*}(Y,1)$ has no cycles.
    Thus, $G-B_G(A\cup Y, 1)$ has no cycles and ${\abs{A\cup Y} \leq 5k}$, as required.
\end{proof}

\section{%
\texorpdfstring{Distance-$d$ packings of two cycles}%
{Distance d-packing of two cycles}}
\label{sec:distancepacking}

In this section, we prove Theorem~\ref{thm:main4}, which is an analogue of Theorem~\ref{thm:main1} for distance-$d$ packings with ${k=2}$.

\distanced*

To prove Theorem~\ref{thm:main4}, we will use two more lemmas.
In the following lemma, we show that if a graph has girth at least~${2d+2}$, then every ball of radius at most~$d$ induces a tree.

\begin{lemma}\label{lem:girth}
    Let~$d$ be a positive integer and let~$G$ be a graph of girth at least~${2d+2}$.
    Then for every vertex~$v$ of~$G$, $G[B_G(v,d)]$ is a tree.
\end{lemma}

\begin{proof}
    Let~$T$ be a BFS spanning tree of $G[B_G(v,d)]$ rooted at~$v$.
    Observe that $T$ has height at most~$d$.
    Therefore, if $G[B_G(v,d)] \neq T$, then $G[B_G(v,d)]$ contains a cycle of length at most~${2d+1}$.
\end{proof}

We now analyse the structure of a graph~$G$ relative to a shortest cycle~$C$ where~${G-B_G(C,d)}$ has no cycles.
Let~$F$ be a rooted tree.
For two distinct vertices~$a$ and~$b$ in~$F$, let~${F(a, b)}$ be the subtree of~$F$ induced by the set of all descendants of~$w$ where~$w$ is the least common ancestor of~$a$ and~$b$ in~$F$. 

\begin{lemma}\label{lem:main}
    Let~$d$ be a positive integer and let~$G$ be a graph of girth at least~${8d+5}$.
    If~$G$ has a shortest cycle~$C$ such that ${G-B_G(C,d)}$ is a forest, then~$G$ contains either a distance-$d$ packing of two cycles or there are sets~$X_1$ and~$X_2$ with~$\abs{X_1} \leq 12(d+1)$ and~$\abs{X_2} \leq 12$ such that both~${G-B_G(X_1,2d)}$ and~${G-B_G(X_2,3d)}$ are forests. 
\end{lemma}

\begin{proof}
    For every vertex~$v$ of~$C$ and every positive integer~${i}$, let~${S(v,i) \coloneqq B_{G-(V(C)\setminus \{v\})}(v, i)}$. 
    For a set~${A \subseteq V(C)}$, let~${S(A,i) \coloneqq \bigcup_{v\in A}S(v,i)}$.
    Since~$G$ has girth at least~${8d+5 \geq 2d+2}$, for every vertex~$v$ of~$C$, we have that~${G-(V(C)\setminus \{v\})}$ also has girth at least~${2d+2}$, and that~${G[S(v,d)]}$ is a tree by Lemma~\ref{lem:girth}. 

    We claim that~$G$ has no short $C$-paths.

    \begin{claim}\label{claim:nolongcpath}
        $G$ has no $C$-path of length at most~${4d+2}$. 
    \end{claim}
    
    \begin{subproof}
        Suppose that such a $C$-path~$Q$ exists.
        Let~$x$ and~$y$ be the ends of~$Q$.
        Since~$C$ has length at least $8d+5$, one of the two $(x,y)$-paths in~$C$, say~$R$, has length at least $4d+3$.
        Then the subgraph obtained from $C\cup Q$ by removing the internal vertices of~$R$ is a cycle shorter than~$C$, contradicting the assumption that~$C$ is a shortest cycle.
    \end{subproof}
    
    By Claim~\ref{claim:nolongcpath} and the girth condition, we have that 
    \begin{itemize}
        \item for distinct vertices~$v$ and~$w$ of~$C$, ${S(v,d) \cap S(w,d) = \emptyset}$ and there is no edge between~$S(v,d)$ and~$S(w,d)$, and 
        \item for every~${v \in B_G(C,d+1) \setminus B_G(C,d)}$, $v$ has a unique neighbour in~$B_G(C,d)$. 
    \end{itemize}
    For each~${v \in B_G(C,d+1) \setminus B_G(C,d)}$, let~$\supp(v)$ be the unique vertex of~$C$ such that $S(\supp(v),d)$ contains a neighbour of~$v$, and let~$R_v$ be the unique ${(v,\supp(v))}$-path of length~${d+1}$ in~$G$.
    Note that~${R_v-v}$ is fully contained in~$S(\supp(v),d)$.

    Let~${F \coloneqq G-B_G(C, d)}$.
    For each component~$F'$ of~$F$, we consider~$F'$ as a rooted tree with an arbitrary fixed root. 
    For every subtree~$F''$ of a component~$F'$ of~$F$, let~$\rt(F'')$ be the vertex of~$F''$ closest in~$F'$ to the root of~$F'$. 
    We consider~$F''$ as a rooted subtree having $\rt(F'')$ as a root.
    For every vertex~$v$ of~$F$, the \emph{level} of~$v$, denoted by~$\lv(v)$, is the distance in~$F$ between~$v$ and the root of the component of~$F$ containing~$v$.

    We will use the following claim.
    
    \begin{claim}\label{claim:ballinforest}
        For $v\in V(F)$, $B_{G}(v, d)\cap V(F)$ has no two vertices $w$ and $z$ in $B_G(C,d+1)$.
    \end{claim}
    \begin{subproof}
        Towards a contradiction, suppose that such vertices $w$ and $z$ exist.
        Since $w,z\in B_{G}(v,d)$ and $\dist_G(v,V(C))>d$, there is a $(w,z)$-path $P$ of length at most $2d$ in $G-V(C)$.
        Then  $R_{w}\cup R_{z}\cup P$ contains a $C$-path of length at most $4d+2$, contradicting Claim~\ref{claim:nolongcpath}.
    \end{subproof}
    
    We now recursively find a sequence $F_1,F_2,\ldots$ of pairwise disjoint subtrees of~$F$, vertices $a_i, b_i\in V(F_i)\cap B_G(C, d+1)$, $w_i\in V(F_i)$, $z_i\in V(F)$, and an $(a_i, b_i)$-path~$P_i$ in~$F_i$ for each~$i$.
    For each $j$, let $L_j\coloneqq B_C\left(\bigcup_{i\in [j]}\{\supp(a_i),\supp(b_i)\}, d\right)$ and $W_j\coloneqq \bigcup_{i\in [j]}\{w_i, z_i\}$. The recursive construction is as follows.
    
    \begin{itemize}
        \item Assume that for $i\in[j-1]$, we have found a subtree $F_i$, vertices $a_i,b_i,w_i,z_i$, and an $(a_i,b_i)$-path $P_i$.
        Note that $L_0=\emptyset$.
        Now, let
        \[
            F_{j-1}^*\coloneqq F-S(L_{j-1},2d)-B_G(W_{j-1},d).
        \] 
        \item If no component of~$F_{j-1}^*$ contains distinct vertices $a,b\in B_G(C, d+1)$ where $\{\supp(a), \supp(b)\}$ is disjoint from~$L_{j-1}$, then we stop the procedure. 
        Otherwise, we choose two vertices~${a_j,b_j \in B_G(C,d+1)}$ such that 
        \begin{enumerate}
            \item $a_j$ and~$b_j$ are contained in the same component~$F'$ of~$F_{j-1}^*$ and \linebreak ${\{\supp(a_j), \supp(b_j)\}\cap L_{j-1} = \emptyset}$; and
            \item subject to (1), $\lv(\rt(F'(a_j, b_j)))$ is maximum.
        \end{enumerate}  
        \item Let ${F_j \coloneqq F'(a_j, b_j)}$, let~$w_j$ be the root of~$F_j$, and let~$P_j$ be the $(a_j, b_j)$-path in~$F_j$.
        If $B_G(w_j, d)\cap V(F)$ contains no vertex in $B_G(C, d+1)$, then let $z_j \coloneqq w_j$.
        Otherwise, by Claim~\ref{claim:ballinforest}, $B_{G}(w_j, d)\cap V(F)$ contains a unique vertex in $B_G(C, d+1)$.
        Let $z_j$ be the vertex.
    \end{itemize}
    
    Let~$t$ be the maximum integer~$j$ for which~$F_j$ is defined.
    We claim that the subtrees~$F_1,\ldots,F_t$ are far from each other in~$F$.
    
    \begin{claim}\label{claim:disjointsubtree}
        For distinct $i,j\in[t]$, we have that $\dist_F(F_i,F_j)\geq d+1$.
    \end{claim}
    
    \begin{subproof}
        Towards a contradiction, suppose that~${\dist_F(F_i, F_j) \leq d}$. 
        Without loss of generality, we may assume that $i<j$.
        Let $Q=u_1u_2\cdots u_m$ be a shortest $(F_i,F_j)$-path in~$F$ with $u_1\in V(F_i)$. 
        Note that~$u_1$ is the unique vertex of~$Q$ contained in~$V(F_i)$. 

        Since~$Q$ has length at most~$d$, if $u_1=w_i$, then $V(Q)\subseteq B_G(w_i,d)\subseteq B_G(W_{j-1}, d)$, contradicting the condition that~$F_j$ avoids $B_G(W_{j-1},d)$. 
        Thus, $u_1\neq w_i$, and therefore~$w_j$ is a descendant of~$w_i$. 
        Since $w_j\neq w_i$, we have that $\lv(w_j)>\lv(w_i)$, contradicting the choice of~$F_i$. 
        Hence, $\dist_F(F_i, F_j)\geq d+1$. 
    \end{subproof}
    
    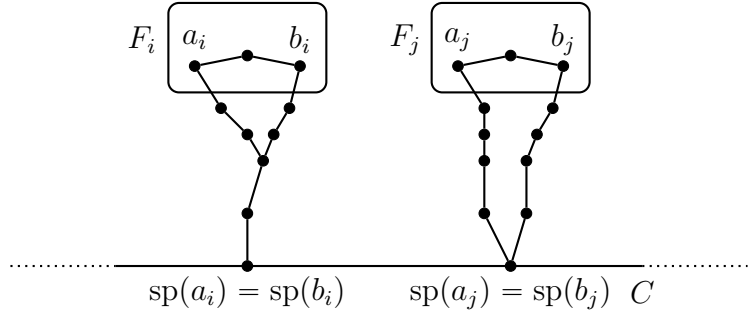
\begin{figure}[t]
        \centering
        \begin{tikzpicture}[scale=0.7]
        \tikzstyle{w}=[circle,draw,fill=black,inner sep=0pt,minimum width=4pt]
        \draw[rounded corners, thick] (2,1)--(12,1);
        \draw[rounded corners, thick, dotted] (12,1)--(14,1);
        \draw[rounded corners, thick, dotted] (0,1)--(2,1);
    
        \node at (12, 0.5) {$C$};
        \draw (4.5, 5) node [w] (v0) {};
        \draw (3.5, 4.8) node [w] (v1) {};
        \draw (4, 4) node [w] (v2) {};
        \draw (4.5, 3.5) node [w] (v3) {};
        \draw (4.8, 3) node [w] (v4) {};
        \draw (4.5, 2) node [w] (v5) {};
        \draw (4.5, 1) node [w] (v6) {};
        \draw (5.5, 4.8) node [w] (w1) {};
        \draw (5.3, 4) node [w] (w2) {};
        \draw (5, 3.5) node [w] (w3) {};
    
        \node at (3.5, 5.3) {$a_i$};
        \node at (5.5, 5.3) {$b_i$};
        \node at (2.5, 5.3) {$F_i$};
        \node at (4.5, 0.5) {$\supp(a_i)=\supp(b_i)$};
        \draw[rounded corners, thick] (3,5.3)--(3,6)--(6,6)--(6, 4.3)--(3,4.3)--(3,5.3);
        \draw[thick] (v1)--(v0)--(w1);
        \draw[thick] (v1)--(v2)--(v3)--(v4)--(v5)--(v6);
        \draw[thick] (w1)--(w2)--(w3)--(v4);
    
        \draw (9.5, 5) node [w] (z0) {};
        
        \draw (8.5, 4.8) node [w] (z1) {};
        \draw (9, 4) node [w] (z2) {};
        \draw (9, 3.5) node [w] (z3) {};
        \draw (9.8, 3) node [w] (z4) {};
        \draw (9.8, 2) node [w] (z5) {};
        \draw (9.5, 1) node [w] (z6) {};
        
        \draw (10.5, 4.8) node [w] (a1) {};
        \draw (10.3, 4) node [w] (a2) {};
        \draw (10, 3.5) node [w] (a3) {};
        
        \draw (9, 3) node [w] (b4) {};
        \draw (9, 2) node [w] (b5) {};
        
        \node at (8.5, 5.3) {$a_j$};
        \node at (10.5, 5.3) {$b_j$};
        \node at (7.5, 5.3) {$F_j$};
        \node at (9.5, 0.5) {$\supp(a_j)=\supp(b_j)$};
        \draw[rounded corners, thick] (8,5.3)--(8,6)--(11,6)--(11, 4.3)--(8,4.3)--(8,5.3);
        
        \draw[thick](z1)--(z2)--(z3)--(b4)--(b5)--(z6);
        \draw[thick](a1)--(a2)--(a3)--(z4)--(z5)--(z6);
        \draw[thick](z1)--(z0)--(a1); 
        \end{tikzpicture}
        \caption{Possible appendages of $H$. Every appendage of $H$ contains one vertex of $C$.}
        \label{fig:appendage}
    \end{figure}

    For each $i\in [t]$, let $M_i\coloneqq P_i\cup R_{a_i}\cup R_{b_i}$.
    Let
    \[
        H \coloneqq C\cup\left(\bigcup_{i\in [t]}M_i\right).
    \]
    Note that for $i\in[t]$, $\supp(a_i)$ may be the same as $\supp(b_i)$.
    In that case, $R_{a_i}$ and~$R_{b_i}$ meet.
    Since $G[S(a_i,d)]$ is a tree, if~$R_{a_i}$ and~$R_{b_i}$ intersect exactly on~$C$, then the vertex in the intersection has degree~$4$ in~$H$, and other vertices of~$M_i$ have degree~$2$ in~$H$.
    For every other case, all vertices of~$M_i$ have degree~$2$ or~$3$ in~$H$.
    See Figure~\ref{fig:appendage} for an illustration.
    
    Let~$J$ be the set of integers $j\in[t]$ such that~$R_{a_j}$ and~$R_{b_j}$ intersect.
    For each $j\in J$, we say that $M_j$ is an \emph{appendage} of~$H$.
    Thus, $H$ can be seen as a graph obtained from the subgraph $C\cup\left(\bigcup_{i\in[t]\setminus J}M_i\right)$, all of whose vertices have degree~$2$ or~$3$, by attaching vertex-disjoint appendages on $C$.
    Clearly, each appendage contains a cycle.

    We first consider the case that~${t < 4}$.
    In this case, we show that~${X_1 \coloneqq L_t \cup W_t}$ and ${X_2 \coloneqq W_t\cup\bigcup_{i \in [t]} \{ \supp(a_i),\supp(b_i) \}}$ are desired sets.
    Since $t<4$, we have that
    \[
        \abs{L_t} + \abs{W_t} 
        \leq (2t) (2d+1) + 2t 
        = 4t(d+1) 
        \leq 12(d+1).
    \]
    and~${\abs{X_2} \leq 4t \leq 12}$.
    
    Let~${G^* \coloneqq G-S(L_t,2d)-B_G(W_t,d)}$.
    To show that~${X_1}$ and~${X_2}$ are as desired, it suffices to show that~$G^*$ is a forest.
    Towards a contradiction, suppose that $G^*$ contains a cycle, say $D$.
    Since $t\geq1$ and the graph obtained from $G[B_G(C,d)]$ by removing one vertex of~$C$ is a forest, $D$ contains a vertex of~$F$.
    Thus, $D$ contains a vertex of
    \[
        F^*\coloneqq F-S(L_t,2d)-B_G(W_t, d).
    \] 

    Let $Y=y_1y_2\cdots y_h$ be a component of $D\cap F^*$.
    As every vertex in $B_G(C,d+1)\setminus B_G(C,d)$ has a unique neighbour in $B_G(C, d)$, we have $h\geq2$.
    Let~$y_0$ be the neighbour of~$y_1$ in~$D$ other than~$y_2$, and let~$y_{h+1}$ be the neighbour of~$y_h$ in~$D$ other than~$y_{h-1}$.
    Observe that $y_0,y_{h+1}\in B_G(C,d)$ and thus, $y_0\in S(\supp(y_1),d)$ and $y_{h+1}\in S(\supp(y_h),d)$.
    If $\supp(y_1)$ or $\supp(y_h)$ is contained in $L_t$, then $\dist_G(L_t,V(Y))\leq d+1\leq2d$, a contradiction.
    Thus, we may assume that neither $\supp(y_1)$ nor $\supp(y_h)$ is contained in $L_t$.
    But then there exists a component of~$F_t^*$ that contains two vertices in~${B_G(C,d+1)}$, contradicting the maximality of~$t$.

    Hence, we conclude that~${X_1}$ and~${X_2}$ are desired sets when~${t<4}$.

    \medskip

    We now consider the case that $t\geq4$.
    In this case, we find a distance-$d$ packing of two cycles.
    To do this, we first find two vertex-disjoint cycles in~$H$.

    \begin{claim}
        \label{claim:manycycles}
        If $t\geq4$, then~$H$ has vertex-disjoint cycles~$C_1$ and~$C_2$ such that for each $i\in[2]$, if~$C_i$ is a cycle contained in some appendage~$R$, then~$C_{3-i}$ is disjoint from~$R$.
    \end{claim}
    \begin{subproof}
        Assume that $t\geq4$.
        If~$H$ contains two appendages, then we are done.
        Thus, we may assume that~$H$ contains at most one appendage.

        Suppose first that~$H$ has no appendages.
        Then every vertex of~$H$ has degree~$2$ or~$3$ in~$H$.
        Since $t\geq4$, $H$ has at least eight vertices of degree~$3$.
        Thus, by Lemma~\ref{lem:twocycles}, it has two disjoint cycles, as required.

        We now suppose that~$H$ has an appendage, say~$R$.
        Let~$H'$ be the graph obtained from~${H-V(R)}$ by recursively removing degree-$1$ vertices.
        Since~$R$ is the unique appendage of~$H$, the number of degree-$3$ vertices of~$H'$ is at least 
        $2(4-1)-2=4$.
        Then~$H'$ contains a cycle~$D$.
        Note that~$D$ is also a cycle in~$H$ and is disjoint from~$R$.
        This proves the claim.
    \end{subproof} 

    By Claim~\ref{claim:manycycles}, $H$ contains vertex-disjoint cycles~$C_1$ and~$C_2$ such that for each $i\in[2]$, if~$C_i$ is a cycle contained in some appendage~$R$, then~$C_{3-i}$ is disjoint from~$R$.
    We show that $\{C_1,C_2\}$ is a distance-$d$ packing of cycles in~$G$.
    
    Let~$Q$ be a shortest $(C_1,C_2)$-path in~$G$.
    Towards a contradiction, suppose that~$Q$ has length at most~$d$. 
    Using the following claims, we argue that~$Q$ has to be contained in~${G-V(C)}$. 
    
    \begin{claim}
        \label{claim:middle}
        If an internal vertex~$w$ of~$Q$ is a vertex of~$C$ such that~${S(w,d)\cap V(C_1\cup C_2)=\emptyset}$, then the neighbours of~$w$ in~$Q$ are contained in~$C$. 
    \end{claim}
    
    \begin{subproof}
        If a neighbour of~$w$ in~$Q$ is contained in $S(w,d)$, then~$Q$ has to have length more than~$d$, a contradiction.
        Thus, the two neighbours of~$w$ in~$Q$ are contained in~$C$.
    \end{subproof}

    \begin{claim}\label{claim:notonC1}
        $Q$ does not contain a vertex in $V(C)\setminus V(C_1\cup C_2)$.
    \end{claim}
    
    \begin{subproof}
        Towards a contradiction, suppose that~$Q$ contains a vertex $w\in V(C)\setminus V(C_1\cup C_2)$.

        Assume first that in~$Q$, every internal vertex~${z \in V(C)}$ satisfies ${S(z,d)\cap V(C_1\cup C_2) = \emptyset}$.
        Since~$w$ is internal, by Claim~\ref{claim:middle}, $Q$ is a path in~$C$ between two vertices in $\bigcup_{i\in [t]}\{\supp(a_i),\supp(b_i)\}$.
        By the construction, the ends of $Q$ are contained in $\{\supp(a_j), \supp(b_j)\}$ for some~${j \in [t]}$, because otherwise it has length more than~$d$.
        Then the ends of~$Q$ are contained in the same cycle, a contradiction.

        Therefore, $Q$ has a vertex~$z$ of~$C$ such that~${S(z,d)\cap V(C_1\cup C_2)\neq\emptyset}$.
        This means that one of~$C_1$ and~$C_2$ is a cycle contained in an appendage of~$H$.
        On the other hand, since $S(z,d)$ cannot contain vertices from both~$C_1$ and~$C_2$, one of the neighbours of~$z$ in~$Q$ is contained in~$C$.
        Therefore, following from $z$ to its neighbour on~$C$, it traverses to some vertex of $\bigcup_{i\in[t]}\{\supp(a_i),\supp(b_i)\}$.
        Since $z$ is also in $\bigcup_{i\in [t]}\{\supp(a_i),\supp(b_i)\}$, $Q$ has length more than $d$, a contradiction. 
    \end{subproof}

    \begin{claim}
        \label{claim:notonC2}
        $Q$ is contained in $G-V(C)$.
    \end{claim}
    
    \begin{subproof}
        By Claim~\ref{claim:notonC1}, it suffices to show that no end of~$Q$ is contained in $(C_1\cup C_2)\cap C$. 

        Towards a contradiction, suppose that an end~$w$ of~$Q$ is contained in~${(C_1 \cup C_2) \cap C}$. 
        Without loss of generality, assume that~$w$ is in~${C_1 \cap C}$. 
        Since~$Q$ is a shortest $(C_1,C_2)$-path, the neighbour of~$w$ in~$Q$, say~$w'$, is not in~$C_1$. 
        By the construction, $w'$ is not in~$C\cap C_2$. 
        Since $Q$ does not contain a vertex in $V(C)\setminus V(C_1\cup C_2)$, $w'$ is contained in~$S(w,d)$. 
        If~$C_2$ has no vertex in $S(w,d)$, then~$Q$ has length longer than~$d$, a contradiction. 
        Thus,~$C_2$ has a vertex in~$S(w,d)$. 
        Since~$w$ is not in~$C_2$, we deduce that~$C_2$ is a cycle contained in an appendage of~$H$. 
        Then by our choice of~$C_1$ and~$C_2$ from Claim~\ref{claim:manycycles}, $C_1$ cannot contain~$w$, a contradiction. 
        Hence, $Q$ is contained in $G-V(C)$. 
    \end{subproof}
    
    Let~${Q = u_1 u_2 \cdots u_m }$.
    By Claim~\ref{claim:notonC2}, there exist distinct integers~${\beta, \gamma \in [t]}$ for which ${u_1 \in V(M_\beta) \setminus V(C)}$ and~${u_m \in V(M_\gamma) \setminus V(C)}$.
    By symmetry, we may assume that~${\beta < \gamma}$.

    Since~${\dist_F(F_\beta, F_\gamma) \geq d+1}$ by Claim~\ref{claim:disjointsubtree}, $Q$ contains some vertex of~$B_G(C,d)$.
    Let~$\eta$ be the minimum integer in~$[m]$ such that~${u_{\eta } \in B_G(C, d)}$.
    The existence of~$u_{\eta}$ implies that if~${u_1 \in V(R_{a_{\beta}} \cup R_{b_{\beta}})}$ or~${u_m \in V(R_{a_{\gamma}} \cup R_{b_{\gamma}})}$, then~$G$ has either a $C$-path of length at most~${3d+1}$ or a cycle of length at most~${3d+1}$.
    This contradicts Claim~\ref{claim:nolongcpath} or the girth condition.
    Thus, we may assume that~${u_1 \in V(P_\beta) \setminus \{ a_\beta, b_\beta \}}$ and~${u_m \in V(P_\gamma) \setminus \{ a_\gamma, b_\gamma\}}$. 

    We divide into two cases depending on the location of~$u_1$.
    
    \medskip
    \noindent\textbf{Case 1: ${u_1 = w_\beta}$.}
    
    In this case, $u_{\eta-1}$ is contained in $B_{G}(w_\beta, d)$.
    By Claim~\ref{claim:ballinforest}, $B_{G}(w_\beta, d)\cap V(F)$ contains a unique vertex having a neighbour in $B_G(C, d)$, and we defined it as $z_{\beta}$.
    Thus, $u_{\eta-1}=z_{\beta}$.
    Then there is a path of length at most~$d$ from~$z_{\beta}$ to~$P_{\gamma}$, contradicting the fact that~$F_{\gamma}$ avoids~$B_G(W_{\gamma-1},d)$.
    
    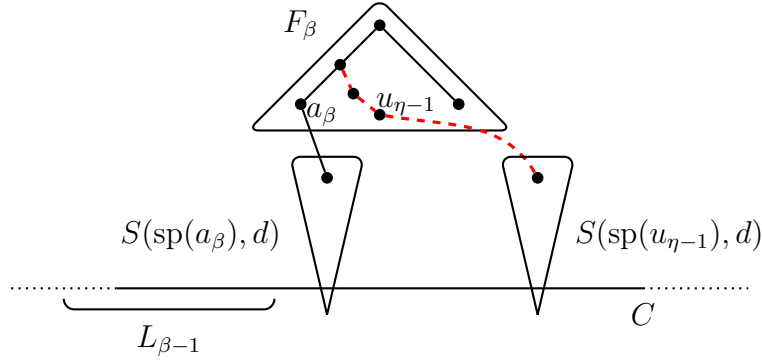
\begin{figure}[t]
        \centering
        \begin{tikzpicture}[scale=0.7]
        \tikzstyle{w}=[circle,draw,fill=black,inner sep=0pt,minimum width=4pt]
        
        \draw[rounded corners, thick] (2,1)--(12,1);
        \draw[rounded corners, thick, dotted] (12,1)--(14,1);
        \draw[rounded corners, thick, dotted] (0,1)--(2,1);
        
        \node at (12, 0.5) {$C$};
        \node at (5.5, 6) {$F_{\beta}$};
        \node at (7.5, 4.5) {$u_{\eta-1}$};
        \node at (12.5, 2) {$S(\supp(u_{\eta-1}), d)$};
        \node at (3.6, 2) {$S(\supp(a_{\beta}), d)$};
        \node at (5.9, 4.3) {$a_{\beta}$};
        
        \draw (7, 6) node [w] (v1) {};
        \draw (5.5, 4.5) node [w] (v2) {};
        \draw (8.5, 4.5) node [w] (v3) {};
        
        \draw (6.25, 5.25) node [w] (w1) {};
        \draw (6.5, 4.7) node [w] (w2) {};
        \draw (7, 4.3) node [w] (w3) {};
        \draw (10, 3.1) node [w] (w4) {};
        \draw (6, 3.1) node [w] (w5) {};
        
        \draw[rounded corners, thick] (6.5,6)--(4.5,4)--(9.5,4)--(7.5,6)--(7,6.5)--(6.5,6);
        \draw[thick] (v2)--(v1)--(v3);
        \draw[very thick, color=red, dashed] (w1)--(w2)--(w3);
        \draw[very thick, color=red, dashed] (w3) to[out=-10,in=120] (w4);
        \draw[rounded corners, thick] (10,0.5)--(10.7,3.5)--(9.3,3.5)--(10,0.5);
        \draw[rounded corners, thick] (6,0.5)--(6.7,3.5)--(5.3,3.5)--(6,0.5);
        
        \draw[rounded corners, thick] (1,0.8)--(1,0.6)--(5,0.6)--(5,0.8);
        \node at (3, 0) {$L_{\beta-1}$};
        \draw[thick](v2)--(w5);
        \end{tikzpicture}
        \caption{When $u_{\eta-1}\in V(F_{\beta})$ in Case~2. 
            The dashed path denotes the path~$Q$. 
            If this happens, then in the $\beta$-th step, we had to select a subtree whose root has level larger than the root of~$F_{\beta}$.} 
        \label{fig:case31}
    \end{figure}

    \medskip
    \noindent\textbf{Case 2: ${u_ 1\in V(P_\beta) \setminus \{w_\beta\}}$.}

    Note that since~${u_1 \in V(C_1)}$, by construction~${V(P_\beta) \subseteq C_1}$. 
    Since the parent of $u_1$ is a vertex of~$P_{\beta}$ and no internal vertex of~$Q$ is contained in~${C_1 \supseteq V(P_{\beta})}$, we deduce that~$u_{\eta-1}$ is a descendant of~$u_1$ in~$F$. 
    If~${\supp(u_{\eta-1}) \in L_{\beta-1}}$, then~$P_\beta$ contains a vertex of~${S(L_{\beta-1}, 2d)}$, a contradiction. 
    Thus, ${\supp(u_{\eta-1}) \notin L_{\beta-1}}$. 
    
    We first consider the subcase when~${u_{\eta-1} \in V(F_{\beta})}$. 
    See Figure~\ref{fig:case31} for an illustration.
    Since~${\supp(u_{\eta-1}) \notin L_{\beta-1}}$, in the $\beta$-th step, we should have taken a subtree of~$F_{\beta}$ containing~$u_{\eta-1}$ and one of~$a_\beta$ and~$b_\beta$ whose root has level larger than $\lv(w_\beta)$, a contradiction. 

    Hence, ${u_{\eta-1} \notin V(F_\beta)}$.
    Since~$u_{\eta-1}$ is a descendant of~$u_1$ in~$F$, by our construction,~$u_{\eta-1}$ is contained in
    \[
        S(L_{\beta-1},2d)\cup B_G(W_{\beta-1},d).
    \]

    \begin{figure}[t]
        \centering
        \begin{tikzpicture}[scale=0.9]
        \tikzstyle{w}=[circle,draw,fill=black,inner sep=0pt,minimum width=4pt]
        
        \draw[rounded corners, thick] (2,1)--(12,1);
        \draw[rounded corners, thick, dotted] (12,1)--(14,1);
        \draw[rounded corners, thick, dotted] (0,1)--(2,1);
        
        \node at (12, 0.5) {$C$};
        \node at (5.5, 9) {$F_{\beta}$};
        \node at (8.6, 6.2) {$u_{\eta-1}$};
        \node at (10.8, 2) {$S(\supp(z_j), d)$};
        \node at (2.8, 2) {$S(\supp(u_{\eta-1}), d)$};
        
        \draw (7, 9) node [w] (v1) {};
        \draw (5.5, 7.5) node [w] (v2) {};
        \draw (8.5, 7.5) node [w] (v3) {};
        
        \draw (6.25, 8.25) node [w] (w1) {};
        \draw (6.5, 7.7) node [w] (w2) {};
        \draw (7, 7.3) node [w] (w3) {};
        \draw (9, 3.1) node [w] (w4) {};
        \draw (5, 3.1) node [w] (w5) {};
        
        \draw (7.8, 6.5) node [w] (z3) {};
        \draw (8, 6.1) node [w] (z4) {};
        
        \draw (8.5, 4.3) node [w] (z1) {};
        \node at (8.95, 4.2) {$w_j$};
        \node at (10.5, 5) {$B_{G}(w_j, d)$};
        \node at (10.5, 6.3) {$B_{G}(z_j, d)$};
        
        \draw (8.5, 5) node [w] (z2) {};
        \node at (9, 5.1) {$z_j$};
        
        \draw[rounded corners, thick] (6.5,6+3)--(4.5,4+3)--(9.5,4+3)--(7.5,6+3)--(7,6.5+3)--(6.5,6+3);
        \draw[rounded corners, thick]
        (4.7,7)--(4.7,4)--(9.3,4)--(9.3,7);
        \draw[rounded corners, thick]
        (9.6, 5.5)--(7, 5.5)--(7, 3.7);
        \draw[rounded corners, thick]
        (9.6, 6.8)--(7.2, 6.8)--(7.2, 5);
        
        \draw[thick] (v2)--(v1)--(v3);
        \draw[very thick, color=red, dashed] (w1)--(w2)--(w3)--(z3)--(z4);
        \draw[rounded corners, thick] (10-1,0.5)--(10.7-1,3.5)--(9.3-1,3.5)--(10-1,0.5);
        \draw[rounded corners, thick] (6-1,0.5)--(6.7-1,3.5)--(5.3-1,3.5)--(6-1,0.5);
        
        \draw[very thick] (z2) to[out=-10,in=60] (w4);
        \draw[very thick, color=red, dashed] (z4) to[out=-180,in=90] (w5);
        \end{tikzpicture}
        \caption{When $u_{\eta-1}\notin V(F_{\beta})$ and $u_{\eta-1}\in B_{G}\left(\bigcup_{i\in [\beta-1]}\{w_i, z_i\}, d\right)$ in Case 2. The dashed path denotes the path $Q$. If this happens, then the ball $B_{G}(z_j,d)$ contains two vertices $z_j$ and $u_{\eta-1}$ having neighbours in $B_G(C, d)$ which makes a $C$-path of length at most $2d+2$.}
        \label{fig:case32}
    \end{figure}
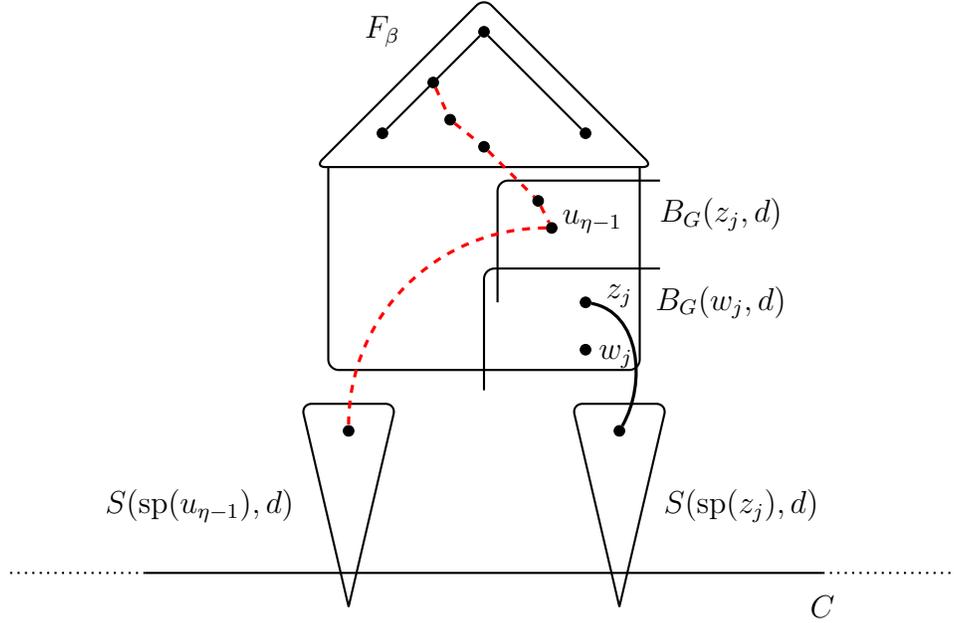

    If~$u_{\eta-1} \in S(L_{\beta-1}, 2d)$, then since~${\supp(u_{\eta-1}) \notin L_{\beta-1}}$, there is a $C$-path of length at most~${3d+1}$, contradicting Claim~\ref{claim:nolongcpath}. 
    Thus, we may assume that~${u_{\eta-1} \in B_{G-V(C)}\left(W_{\beta-1}, d\right)}$. 
    See Figure~\ref{fig:case32} for an illustration.

    Assume that~${u_{\eta-1} \in B_{G}(\{w_j, z_j\}, d)}$ for some~${j \in [\beta-1]}$. 
    If~${w_j = z_j}$, then~$u_{\eta-1}$ is the unique vertex in $B_{G}(w_j, d)$ having a neighbour in $B_G(C, d)$, and we define it as~${z_j = u_{\eta-1}}$. 
    Thus, ${w_j \neq z_j}$. 
    If~${u_{\eta-1} \in B_{G}(w_j, d) \setminus B_{G}(z_j, d)}$, then $B_{G}(w_j, d)$ contains two vertices, namely $u_{\eta-1}$ and $z_j$, having neighbours in $B_G(C, d)$, contradicting Claim~\ref{claim:ballinforest}. 
    Thus,~${u_{\eta-1} \in B_{G}(z_j, d)}$. 
    Since~${u_1 \notin B_{G}(z_j, d)}$ and~${\eta-1 \leq d}$, $u_{\eta-1}$ and~$z_j$ are distinct vertices in~${B_{G}(z_j, d)}$ having neighbours in~$B_G(C, d)$, again contradicting Claim~\ref{claim:ballinforest}. 

    Hence, we conclude that~${\{C_1,C_2\}}$ is a distance-$d$ packing of cycles in~$G$ when~${t \geq 4}$.

    \medskip
    
    Therefore, if~${t < 4}$, then~$G$ has the desired sets~$X_1$ and~$X_2$ of at most~${12(d+1)}$ and~$12$ vertices, respectively, such that~${G-B_G(X_1,2d)}$ and~${G-B_G(X_2,3d)}$ are forest, and otherwise~$G$ has a distance-$d$ packing of two cycles.
    This completes the proof. 
\end{proof}

Now, we are ready to show Theorem~\ref{thm:main4}.

\begin{proof}[Proof of Theorem~\ref{thm:main4}]
    Suppose that~$G$ has no distance-$d$ packing of two cycles.
    If~$G$ is a forest, then we are done by taking~${X_1 \coloneqq \emptyset \eqqcolon X_2}$. 
    Thus, we may assume that~$G$ has a cycle. 

    Let~$C$ be a shortest cycle of~$G$.
    As~$G$ has no distance-$d$ packing of two cycles, ${G-B_G(V(C),d)}$ is a forest. 
    If~$C$ has length less than~${8d+5}$, then we are done by taking~${X_1 \coloneqq V(C)}$ and~$X_2$ to be a set of~$4$ vertices in~$V(C)$ such that~${B_C(X_2,d) = V(C)}$. 
    Thus, we may assume that~$C$ has length at least $8d+5$, that is, $G$ has girth at least~$8d+5$.
    Then by Lemma~\ref{lem:main}, $G$ has sets~$X_1$ and~$X_2$ of at most $12(d+1)$ and~$12$ vertices, respectively, such that both ${G-B_G(X_1,2d)}$ and~${G-B_G(X_2,3d)}$ are forest.
\end{proof}

\section{Bounding tree-independence number}
\label{sec:tin}

We start with a proof of Corollary~\ref{cor:K1tfree}.

\tinbound*

\begin{proof}
    We first apply Theorem~\ref{thm:main1}.
    We may assume that we found a set~$X$ of at most ${f(k,\ell)=\mathcal{O}(\ell k\log k)}$ vertices such that $G'\coloneqq G-B_G(X,1)$ has no $\ell$-cycle, because otherwise we are done.
    Since~$G'$ has no $\ell$-cycle, by Lemma~\ref{lem:birmele}, one can find a tree-decomposition $(T,\beta)$ of~$G$ of width at most~$\ell-2$ in linear time.
    
    For each $u\in V(T)$, let $\beta'(u)\coloneqq\beta(u)\cup B_G(X,1)$.
    It is readily seen that $(T,\beta')$ is a tree-decomposition of~$G$.

    We now show that $(T,\beta')$ has independence number $\mathcal{O}(t\ell k\log k)$.
    Since $\abs{\beta(u)}\leq\ell-1$ for every $u\in V(T)$, it suffices to show that $G[B_G(X,1)]$ has independence number at most~$f(k,\ell)(t-1)$.
    Suppose not.
    By the pigeonhole principle, there exists $v\in X$ such that $G[B_G(v,1)]$ has independence number at least~$t$.
    Since $v$ is adjacent to every other vertex in $B_G(v,1)$, $G[B_G(v,1)]$ has $K_{1,t}$ as an induced subgraph, a contradiction.

    Hence, $(T,\beta')$ has independence number at most $f(k,\ell)(t-1)+\ell-1=\mathcal{O}(t\ell k\log k)$.
\end{proof}

We now present a proof of Corollary~\ref{cor:K1tfreed}. 

\tinboundd*

\begin{proof}
    We apply Theorem~\ref{thm:main4}.
    We may assume that we found a set~$X$ of at most~${12(d+1)}$ vertices such that~${G-B_G(X,2d)}$ is a forest, because otherwise we are done. 
    Since
    \[
        \abs{B_G(X,2d)}\leq12(d+1)\cdot\left(\frac{t\cdot((t-1)^{2d}-1)}{t-2}+1\right)=\mathcal{O}(d\cdot t^{2d}),
    \]
    the statement holds. 
\end{proof}

As mentioned in the introduction, 
we show that the class of $K_{1,3}$-free graphs without a distance-$2$ packing of two cycles has unbounded tree-independence number. We use the following result. For a graph $G$, the \emph{line graph} of $G$, denoted by $L(G)$, is the graph on vertex set $E(G)$ such that for distinct $e,f\in E(G)$, they are adjacent in $L(G)$ if and only if they share an end in $G$.

\begin{theorem}[Dallard et al.~\cite{DallardKKMMW2024}]\label{biclique}
    For all positive integers $m$ and $n$ with $m \leq n$, $\tin(L(K_{m,n})) = m$.
\end{theorem}

\unboundedtin*

\begin{proof}
   It is well known that line graphs are $K_{1,3}$-free (see~\cite{BrandstadtLS1999}). Thus, it is sufficient to show that for every positive integer~$n$, $L(K_{n,n})$ has no distance-$2$ packing of two cycles.  This follows from the fact that for every two vertices $e,f\in V(L(K_{n,n}))$, $\dist_{L(K_{n,n})}(e,f)\le 2$, because either they share an end in~$K_{n,n}$ or there is an edge~$g$ sharing ends with both~$e$ and~$f$.
\end{proof}

\section{Discussion}
\label{sec:conclusion}

We now present several open problems regarding our results.
The first one is about the tightness of the function $f(k,\ell)$ in Theorem~\ref{thm:main1}.
Recall that the following conjecture is true for a constant~$\ell$.

\begin{conjecture}\label{conj:long}
    There exists a function ${f(k,\ell) = \mathcal{O}(k\ell+ k\log k)}$ such that for all integers ${k\geq1}$ and ${\ell\geq3}$, every graph~$G$ contains either an induced packing of~$k$ cycles of length at least~$\ell$ or a set~$X$ of at most~$f(k,\ell)$ vertices such that~${G-B_G(X,1)}$ has no cycle of length at least~$\ell$. 
\end{conjecture}

We remark that there is no function $f(k,\ell)=o(k\ell+k\log k)$ that makes Conjecture~\ref{conj:long} true.
Mousset et al.~\cite{MoussetNSW17} showed that every graph contains either~$k$ vertex-disjoint $\ell$-cycles or a set of at most $f(k,\ell)=O(k\ell+k\log k)$ vertices which hits all $\ell$-cycles, and that the function~$f$ is best possible up to a multiplicative constant.
For an arbitrary graph~$G$, let~$H$ be the graph obtained from~$G$ by subdividing each edge twice.
Then every set of vertex-disjoint cycles in~$H$ is an induced packing of cycles in~$H$.
Note that~$H$ has an induced packing of~$k$ $(3\ell)$-cycles if and only if~$G$ has~$k$ vertex-disjoint $\ell$-cycles.
In addition, the minimum size of a set $X\subseteq V(G)$ which hits all $\ell$-cycles in~$G$ is equal to the minimum size of a set $X'\subseteq V(H)$ such that $H-B_H(X',1)$ has no $(3\ell)$-cycle.
Hence, there is no function $f(k,\ell)=o(k\ell+k\log k)$ that makes Conjecture~\ref{conj:long} true.

As generalisations of Theorem~\ref{thm:main1}, one may restrict the cycles to be packed to satisfy various constraints. 
The following constraint considers cycles containing a vertex contained in a fixed set of vertices.
For a graph~$G$ and a set~${S \subseteq V(G)}$, an \emph{$S$-cycle} in~$G$ is a cycle containing a vertex in~$S$. 

\begin{conjecture}
    There exists a function~${f(k) = \mathcal{O}(k\log k)}$ such that for every positive integer $k$, every graph~$G$, and every $S\subseteq V(G)$, $G$ contains either an induced packing of~$k$ $S$-cycles or a set~$X$ of at most~$f(k)$ vertices such that $G-B_G(X,1)$ has no $S$-cycles.
\end{conjecture}

We can also consider constraints in terms of the parity of the length of cycles, for example cycles of even length. 

\begin{conjecture}
    There exists a function~${f(k) = \mathcal{O}(k\log k)}$ such that for every positive integer~${k}$, every graph~$G$ contains either an induced packing of~$k$ cycles of even length or a set~$X$ of at most~$f(k)$ vertices such that~${G-B_G(X,1)}$ has no cycles of even length.
\end{conjecture}

It is well known that cycles of odd length do not satisfy such a duality even in the non-induced setting \cite{Reed99}.
Gollin et al.~\cite{GollinHKOY2022+} characterised for which positive integers~$m$ and non-negative integers~${\ell < m}$, there exists a function~$f$ such that every graph either contains a packing of~$k$ vertex-disjoint cycles of length~$\ell \bmod{m}$ or a set of size at most~$f(k)$ that hits all such cycles. 
We ask whether such a result extends to the induced setting.

\begin{question}\label{quest:modulo}
    For which positive integers~$m$ and non-negative integers~${\ell < m}$ is there a function~$f$ such that for every positive integer~${k}$, every graph~$G$ contains either an induced packing of~$k$ cycles of length~${\ell\bmod{m}}$ or a set~$X$ of at most~$f(k)$ vertices such that $G-B_G(X,1)$ has no cycles of length $\ell\bmod{m}$?
\end{question}

Let us finish by discussing these type of questions in the context of induced minors and induced topological minors.
For a graph~$H$ and a positive integer~$k$, let~$kH$ denote the disjoint union of~$k$ copies of~$H$. 
A fundamental result of graph minor theory by Robertson and Seymour \cite{RobertsonS1986} states that a graph~$H$ is planar if and only if there exists a function~$f$ such that every graph~$G$ either contains $kH$ as a minor or a set~${X \subseteq V(G)}$ of size at most~${f(k)}$ such that~${G - X}$ is $H$-minor-free. 
Clearly, the fact that for non-planar graphs~$H$ such a duality does not hold extends to induced minors. 
For planar graphs, we conjecture the following. 

\begin{conjecture}
    \label{conj:planarinducedminorEP}
    For each planar graph~$H$, there exist a function~$f$ and a constant~$d$ such that for every positive integer~${k}$, every graph~$G$ contains either $kH$ as an induced minor, or a set~$X$ of at most $f(k)$ vertices such that~${G-B_G(X,d)}$ is $H$-induced-minor-free. 
\end{conjecture}

Theorem~\ref{thm:main1} proves the special case of~${H = K_3}$.
Next, we show that Conjecture~\ref{conj:planarinducedminorEP} holds when the host graph has bounded tree-independence number. 
The proof operates along the lines of the standard proof of the Erd\H{o}s-P\'{o}sa property for minors in the regime of graphs of bounded tree-width, see for example~\cite[Theorem 12.6.5]{Diestel2017}. 

\begin{theorem}
    \label{thorem:EPinducedMinorboundedtreealpha}
    There exists a function~$f(k,w) = \mathcal{O}(w k \log k)$ such that for every pair of positive integers~$k$ and~$w$, every graph~$G$ with tree-independence number at most~$w$ either contains $kH$ as an induced minor or a set~$X$ of at most~$f(k,w)$ vertices such that~${G - B_G(X,1)}$ is $H$-induced-minor-free. 
\end{theorem}

\begin{proof}
    We define~${f(1,w) \coloneqq 0}$ and~${f(k,w) \coloneqq 2f(k-1,w) + w}$. 

    Let~$G$ be a graph with tree-independence number~$w$ and let~$(T,\beta)$ be a tree-decomposition of~$G$ with independence number at most~${w}$. 
    For every edge~${e = t_1t_2 \in E(T)}$ and each~${i \in [2]}$, let~$T_{e,t_i}$ denote the the component of~${T - e}$ containing~$t_i$ and let~$G_{e,t_i}$ denote the graph~${G\left[ \bigcup_{t \in V(T_{e,t_i})}\beta(t) \right] - \beta(t_{3-i})}$. 

    If there is an edge~${e = t_1t_2}$ for which both $G_{e,t_1}$ and~$G_{e,t_2}$ are $H$-induced-minor-free, then let~$X$ be a maximal independent set in~${G[\beta(t_1) \cap \beta(t_2)]}$. 
    Observe that since~$B_G(X,1)$ contains~${\beta(t_1) \cap \beta(t_2) = V(G)\setminus V(G_{e,t_1} \cup G_{e,t_2})}$, we obtain that~${G - B_G(X,1)}$ is $H$-induced-minor-free. 
    
    If there is an edge~${e = t_1t_2}$ such that both~$G_{e,t_1}$ and~$G_{e,t_2}$ contains $H$ as an induced minor, then, if at least one of them contains $(k-1)H$ as an induced minor, then~$G$ contains $kH$ as an induced minor. 
    Otherwise, for each~${i \in [2]}$, by induction~$G_{e,t_i}$ contains a set~$X_i$ of size~$f(k-1,w)$ such that~${G_{e,t_i} - B_{G_{e,t_i}}(X_i,1)}$ is $H$-induced-minor-free. 
    Let~$Y$ be an independent set in~${G[\beta(t_1) \cap \beta(t_2)]}$. 
    Since~$B_G(Y,1)$ contains all vertices of~$G$ that are not contained in~${G_{e,t_1} \cup G_{e,t_2}}$, we conclude that~${X \coloneqq X_1 \cup Y \cup X_2}$ is a set of size at most~${2 f(k-1,w) + w}$ such that~${G - B_G(X,1)}$ is $H$-induced-minor-free, as desired. 

    So we may assume that for each edge~${e = t_1t_2}$ of~$T$ precisely one of~$G_{e,t_1}$, $G_{e,t_2}$ contains~$H$ as an induced minor. 
    By orienting the edges of~$T$ towards this end, we find a node~${t \in V(T)}$ such that for all~${e = st \in E(T)}$ incident to~$t$, we have~$G_{e,s}$ is $H$-induced-minor-free. 
    Let~$X$ be an independent set in~$G[\beta(t)]$. 
    Note that~$B_G(X,1)$ contains~$G[\beta(t)]$. 
    Hence, ${G - B_G(X,1)}$ is $H$-induced-minor-free, as desired. 
\end{proof}

Robertson and Seymour's result is similarly proved in graphs of bounded tree-width. 
Since every planar graph is a (induced) minor of some grid, they finish the proof using their grid theorem, which states that every graph of large tree-width has a large grid as a minor. 

Similarly, Korhonnen's grid induced minor theorem~\cite{Korhonen2023} allows us to deduce a version of Conjecture~\ref{conj:planarinducedminorEP} where we restrict the class of host graphs to graphs of bounded maximum degree. 
Dallard et al.~\cite{DallardKKMMW2024} conjectured that a hereditary class of $K_{1,t}$-free graphs of large tree-independence number contains a large grid as an induced minor. 
Since $K_{1,t}$-free graphs with no $kK_3$ induced minors have bounded tree-independence number (Corollary~\ref{cor:K1tfree}), it can be seen as evidence for this conjecture. 
Moreover, if true, this conjecture would imply a version of Conjecture~\ref{conj:planarinducedminorEP} where we restrict the class of host graphs to $K_{1,t}$-free graphs. 
Dallard et al.~\cite[Theorem 7.3]{DallardMS2024} characterised the graphs~$H$ such that the class of $H$-induced-minor-free graphs has unbounded tree-independence number. 
In particular, a graph of large tree-independence number need not contain a large grid as an induced minor. 
As such, in order to resolve Conjecture~\ref{conj:planarinducedminorEP}, a different approach will be necessary. Note that since a graph has $P_n$ as an induced minor if and only if it has $P_n$ as an induced subgraph,  Conjecture~\ref{conj:planarinducedminorEP} holds for all paths, while the class of $P_n$-induced-minor-free graphs has unbounded tree-independence number for $n\ge 4$.

A graph~$H$ is an \emph{induced topological minor} of a graph~$G$ if~$G$ has an induced subgraph isomorphic to a subdivision of~$H$.
Kim and Kwon~\cite{KimK20} asked for which graphs~$H$ is there a function~$f$ such that for every positive integer~$k$, every graph contains either~$k$ vertex-disjoint induced topological minors of~$H$ or a set of at most~$f(k)$ vertices meeting all induced topological minors of~$H$. They showed that such a function exists for $H=C_4$, but does not exist for longer cycles, and Kwon and Raymond~\cite{KwonR2021} analysed more graphs. We ask a similar question for induced packings. Theorem~\ref{thm:main1} also proves the special case of~${H = K_3}$.

\begin{question}
    \label{ques:planarinducedsubdivEP}
    For which graphs~$H$ is there a function~$f$ and a constant~$d$ such that for every positive integer~${k}$, every graph~$G$ contains either $kH$ as an induced topological minor, or a set~$X$ of at most $f(k)$ vertices such that~${G-B_G(X,d)}$ has no $H$ as an induced topological minor?
\end{question}

\printbibliography

\end{document}